\documentclass[12pt]{article}

\usepackage[latin1]{inputenc}
\usepackage{fullpage}
\usepackage{amsfonts}
\usepackage{amssymb}
\usepackage{amsmath}
\usepackage{stmaryrd}  
\usepackage{wasysym}

\usepackage{multirow} 

\usepackage[colors]{optsys}
\newcommand{\subsetextr}{\widetilde{\TripleNormSphere}}

\renewcommand{\bleue}[1]{#1}

\title{Orthant-Strictly Monotonic Norms,\\
  Generalized Top-$k$ and $k$-Support Norms\\
  and the $\ell_0$ Pseudonorm}

\author{Jean-Philippe Chancelier 
  and Michel De Lara\footnote{michel.delara@enpc.fr}
  \\ CERMICS, Ecole des Ponts, Marne-la-Vall\'ee, France}

\begin{document}

\maketitle

\begin{abstract}
  The so-called $\ell_0$~pseudonorm on the Euclidean space~\( \mathbb{R}^d \)
  counts the number of nonzero components of a vector.
  We say that a sequence of norms is
  strictly increasingly graded (with respect to the $\ell_0$~pseudonorm)
  if it is nondecreasing and that the sequence of norms of a vector~$x$  
  becomes stationary exactly at the index~$\ell_0(x)$. 
  In this paper, 
  with any (source) norm, we associate sequences of 
  generalized top-$k$ and $k$-support norms,
  and we also introduce the new class of orthant-strictly monotonic norms
  (that encompasses the $\ell_p$ norms, but for the extreme ones).
  Then, we show that an 
  orthant-strictly monotonic source norm
  generates a sequence of generalized top-$k$ norms
  which is strictly increasingly graded.
  With this, we provide a systematic way to generate 
  sequences of norms with which the level sets of the $\ell_0$~pseudonorm
  are expressed by means of the difference of two norms. 
  Our results rely on the study of orthant-strictly monotonic norms.
\end{abstract}





\section{Introduction}

The \emph{counting function}, also called \emph{cardinality function}
or \emph{\lzeropseudonorm}, 
counts the number of nonzero components of a vector in~$\RR^d$.
The \lzeropseudonorm\ shares three out of the four axioms of a norm ---
nonnegativity, positivity except for \( \primal =0 \), subadditivity ---
but the \lzeropseudonorm\ is 0-homogeneous (hence the axiom of 1-homogeneity
does not hold true).
The \lzeropseudonorm\ is used in sparse optimization, either as criterion or in the
constraints, to obtain solutions with few nonzero entries.
The \lzeropseudonorm\ is nonconvex, 
but it has been established that its level sets 
can be expressed by means of the difference between two 
convex functions, more precisely two norms,
taken from the nondecreasing sequence of so-called
top-$k$ norms (see~\cite{Tono-Takeda-Gotoh:2017}
and references therein).
In this paper, we generalize this kind of result to a large class
of sequences of norms by introducing three concepts and by relating them to the \lzeropseudonorm.

First, we define
sequences of generalized top-$k$ and $k$-support norms,
associated with any (source) norm on~$\RR^d$.
This extends already known concepts of top-$k$ and $k$-support norms
\cite{Argyriou-Foygel-Srebro:2012,Obozinski-Bach:hal-01412385}. 
Second, we introduce a new class of 
orthant-strictly monotonic norms on~$\RR^d$.
We rely on the notion of orthant-monotonic norm\footnote{%
  It is proved in~\cite[Lemma~2.12]{Gries:1967} that a norm is orthant-monotonic if and only if it is
  monotonic in every orthant, hence the name.}
introduced and studied in~\cite{Gries:1967,Gries-Stoer:1967}
with further developments in~\cite{Marques_de_Sa-Sodupe:1993}.
With such an orthant-strictly monotonic norm, when one component of a vector moves away from zero,
the norm of the vector strictly grows. 
Thus, an orthant-strictly monotonic norm is sensitive to the support 
of a vector, like the \lzeropseudonorm.
We study this class of norms, using the notions of
dual vector pair for a norm
\cite{Gries:1967,Gries-Stoer:1967,Marques_de_Sa-Sodupe:1993}
(refered to as polar alignment in~\cite{Fan-Jeong-Sun-Friedlander:2020}),
and of Birkhoff orthogonality~\cite{Birkhoff:1935},
and strict Birkhoff orthogonality~\cite{Sain-Paul-Jha:2015}. 
Third, we define sequences of norms that are
strictly increasingly graded (with respect to the $\ell_0$~pseudonorm):
the sequence of norms of a vector~$\primal$  
is nondecreasing and becomes stationary exactly at the index~$\lzero(\primal)$. 
Thus equipped, we show  why and how these three concepts
prove especially relevant for the \lzeropseudonorm.

\bleue{%
  This paper has some parts in common with the paper
\cite{Chancelier-DeLara:2021_SVVA}.
Indeed, the paper \cite{Chancelier-DeLara:2021_SVVA} 
built upon \cite{Chancelier-DeLara:2021_ECAPRA_JCA,Chancelier-DeLara:2022_CAPRA_OPTIMIZATION}
to prove hidden convexity of any nondecreasing function of the \lzeropseudonorm,
using conjugacies based on a class of norms that were not considered in
\cite{Chancelier-DeLara:2021_ECAPRA_JCA,Chancelier-DeLara:2022_CAPRA_OPTIMIZATION},
the orthant-strictly monotonic norms.
This is why, we needed specific results on orthant-strictly monotonic norms,
and provided them\footnote{%
  More precisely, \cite[Proposition~12]{Chancelier-DeLara:2021_SVVA}
corresponds to Item~\ref{it:ICS} in Proposition~\ref{pr:orthant-monotonic},
\cite[Proposition~13]{Chancelier-DeLara:2021_SVVA}
corresponds to Item~\ref{it:SICS} and Item~\ref{it:SDC}
in Proposition~\ref{pr:orthant-strictly_monotonic},
\cite[Proposition~15]{Chancelier-DeLara:2021_SVVA}
corresponds to the second Item in Proposition~\ref{pr:increasingly_graded}.}
in \cite[Appendix~2]{Chancelier-DeLara:2021_SVVA}.
However, the current paper deals with different issues.
Indeed, we focus here on a thorough characterization of
orthant and orthant-strictly monotonic norms, 
and on the properties of derived sequences of norms.
The only connection with the \lzeropseudonorm\ is in the notion
of (strictly) increasingly graded norms and how this allows to express 
the level sets of the \lzeropseudonorm\ 
by means of the difference between two norms.
This last question was not treated in  \cite{Chancelier-DeLara:2021_SVVA}.
}

The paper is organized as follows.
In Sect.~\ref{Orthant-monotonic_and_orthant-strictly_monotonic_norms}
we introduce a new class of 
orthant-strictly monotonic norms on~$\RR^d$,
for which we provide different characterizations.
In Sect.~\ref{Generalized_top-k_and_k-support_norms}, 
we define sequences of generalized top-$k$ and $k$-support norms,
generated from a source norm, and we study their properties, be they general or
under orthant-monotonicity.
Finally, in
Sect.~\ref{The_lzeropseudonorm_orthant-monotonicity_and_generalized_top-k_and_k-support_norms}
we introduce the notion of sequences of norms
that are (strictly) increasingly graded with respect to the \lzeropseudonorm.
We show that an orthant-strictly monotonic source norm 
generates a sequence of generalized top-$k$ norms
which is strictly increasingly graded with respect to the \lzeropseudonorm.
We also study the  sequence of generalized $k$-support norms.
In conclusion, we hint at possible applications in sparse optimization.

\section{Orthant-monotonic and orthant-strictly monotonic norms}
\label{Orthant-monotonic_and_orthant-strictly_monotonic_norms}

In~\S\ref{Background_on_norms}, we recall well-known definitions for norms.
In~\S\ref{Orthant-monotonic_norms}, we provide new characterizations of orthant-monotonic norms.
Then, in~\S\ref{Orthant-strictly_monotonic_norms},
we introduce the new notion of orthant-strictly monotonic norm, 
and we provide characterizations, as well as properties,
that will prove especially relevant for the \lzeropseudonorm.
\medskip

\subsection{Background on norms}
\label{Background_on_norms}

We work on the Euclidean space~$\RR^d$
(where~$d$ is a nonzero integer), equipped with the scalar product 
\( \proscal{\cdot}{\cdot} \) (but not necessarily with the Euclidean norm).
Thus, all norms define the same (Borel) topology.
We use the notation \( \ic{j,k}=\na{j, j+1,\ldots,k-1,k} \) for any pair of
integers such that \( j \leq k \). 
For any vector~\( \primal \in \RR^d \), we define its \emph{support} by 
\begin{equation}
  \Support{\primal} = \bset{ j \in \ic{1,d} }%
  {\primal_j \not= 0 } \subset \ic{1,d}
  \eqfinp
  \label{eq:support_of_a_vector}
\end{equation}
For any norm~$\TripleNorm{\cdot}$ on~$\RR^d$,
we denote the unit sphere and the unit ball 
of the norm~$\TripleNorm{\cdot}$ by 
\begin{subequations}\label{eq:triplenorm_unit_sphere_ball}
  \begin{align}
    \TripleNormSphere
    &  = 
      \bset{\primal \in \RR^d}{\TripleNorm{\primal} = 1} 
      \eqfinv
      \label{eq:triplenorm_unit_sphere}
    \\
    \TripleNormBall 
    &  = 
      \bset{\primal \in \RR^d}{\TripleNorm{\primal} \leq 1} 
      \eqfinp
      \label{eq:triplenorm_unit_ball}
  \end{align}
\end{subequations}

\subsubsection*{Dual norms}

We recall that the following expression 
\begin{equation}
  \TripleNorm{\dual}_\star = 
  \sup_{ \TripleNorm{\primal} \leq 1 } \proscal{\primal}{\dual} 
  \eqsepv \forall \dual \in \RR^d
  \label{eq:dual_norm}
\end{equation}
defines a norm on~$\RR^d$, 
called the \emph{dual norm} \( \TripleNormDual{\cdot} \)
\cite[Definition~6.7]{Aliprantis-Border:1999}.
In the sequel, we will occasionally consider the $\ell_p$-norms~$\norm{\cdot}_{p}$ on the space~\( \RR^d \),
defined by \( \norm{\primal}_{p} = \bp{\sum_{i=1}^d |\primal_i|^p}^{\frac{1}{p}}\) 
for $p\in [1,\infty[$, and by
\(\norm{\primal}_{\infty} = \sup_{i\in\ic{1,d}} |\primal_i|\).
It is well-known that the dual norm of the norm~$\norm{\cdot}_{p}$
is the $\ell_q$-norm~$\norm{\cdot}_{q}$, where $q$ is such that \(1/p + 1/q=1\) 
(with the extreme cases $q=\infty$ when $p=1$, and $q=1$ when $p=\infty$). 

We denote the unit sphere and the unit ball 
of the dual norm~$\TripleNormDual{\cdot}$ by 
\begin{subequations}
  \begin{align}
    \TripleNormDualSphere
    &  = 
      \defset{\dual \in \RR^d}{\TripleNormDual{\dual} = 1} 
      \eqfinv
      \label{eq:triplenorm_Dual_unit_sphere}
    \\
    \TripleNormDualBall
    &  = 
      \defset{\dual \in \RR^d}{\TripleNormDual{\dual} \leq 1} 
      \eqfinp
      \label{eq:triplenorm_Dual_unit_ball}
  \end{align}
\end{subequations}
For any subset \( \Primal\subset\RR^d \),
\( \sigma_{\Primal} : \RR^d \to [-\infty,+\infty] \) denotes the 
\emph{support function of the subset~$\Primal$}:
\begin{equation}
  \sigma_{\Primal}\np{\dual} = 
  \sup_{\primal\in\Primal} \proscal{\primal}{\dual}
  \eqsepv \forall \dual \in \RR^d
  \eqfinp
  \label{eq:support_function}
\end{equation}
\begin{subequations}
  It is easily established that 
  \begin{equation}
    \TripleNorm{\cdot} = \sigma_{\TripleNormDualBall} = \sigma_{\TripleNormDualSphere} 
    \mtext{ and } 
    \TripleNormDual{\cdot} = \sigma_{\TripleNormBall} = \sigma_{\TripleNormSphere}
    \eqfinv
    \label{eq:norm_dual_norm}
  \end{equation}
  where \( \TripleNormDualBall \), the unit ball of the dual norm,
  is the polar set~\( \TripleNormBall^{\odot} \) 
  of the unit ball~\( \TripleNormBall \):
  \begin{equation}
    \TripleNormDualBall
    =\TripleNormBall^{\odot} 
    = \defset{\dual \in \RR^d}{\proscal{\primal}{\dual} \leq 1 
      \eqsepv \forall \primal \in \TripleNormBall } 
    \eqfinp
    \label{eq:norm_dual_norm_unit_ball}
  \end{equation}
  Since the set \( \TripleNormBall \) is closed, convex and contains~$0$,
  we have~\cite[Theorem~5.103]{Aliprantis-Border:1999}
  \begin{equation}
    \TripleNormBall^{\odot\odot} = \bp{ \TripleNormBall^{\odot} }^{\odot} 
    = \TripleNormBall 
    \eqfinv 
    \label{eq:bipolarball=ball}
  \end{equation}
  hence the \emph{bidual norm} 
  \( \TripleNorm{\cdot}_{\star\star}= \bp{ \TripleNorm{\cdot}_\star }_\star \)
  is the original norm:
  \begin{equation}
    \TripleNorm{\cdot}_{\star\star}= \bp{ \TripleNorm{\cdot}_\star }_\star 
    = \TripleNorm{\cdot} \eqfinp
    \label{eq:bidual_norm=norm}
  \end{equation}
\end{subequations}

\subsubsection*{\( \TripleNorm{\cdot} \)-duality}

By construction of the dual norm in~\eqref{eq:dual_norm}, we have 
the inequality
\begin{subequations}
  \begin{equation}
    \proscal{\primal}{\dual} \leq 
    \TripleNorm{\primal} \times \TripleNormDual{\dual} 
    \eqsepv \forall  \np{\primal,\dual} \in \RR^d \times \RR^d 
    \eqfinp 
    \label{eq:norm_dual_norm_inequality}
  \end{equation}
  One says that \( \dual \in \RR^d \) is 
  \emph{\( \TripleNorm{\cdot} \)-dual} to \( \primal \in \RR^d \),
  denoted by \( \dual \parallel_{\TripleNorm{\cdot}} \primal \),
  if equality holds in Inequality~\eqref{eq:norm_dual_norm_inequality},
  that is,
  \begin{equation}
    \dual \parallel_{\TripleNorm{\cdot}} \primal
    \iff
    \proscal{\primal}{\dual} =
    \TripleNorm{\primal} \times \TripleNormDual{\dual} 
    \eqfinp 
    \label{eq:couple_TripleNorm-dual}
  \end{equation}
\end{subequations}
The terminology \( \TripleNorm{\cdot} \)-dual comes from
\cite[page~2]{Marques_de_Sa-Sodupe:1993}
(see also the vocable of \emph{dual vector pair} in~\cite[Equation~(1.11)]{Gries:1967}
and of \emph{dual vectors} in~\cite[p.~283]{Gries-Stoer:1967}, 
whereas it is refered as \emph{polar alignment}
in~\cite{Fan-Jeong-Sun-Friedlander:2020}).

We illustrate the \( \TripleNorm{\cdot} \)-duality
in the case of the $\ell_p$-norms~$\norm{\cdot}_{p}$, for $p\in [1,\infty]$.
The notation $\primal\circ\primal' =
\np{ \primal_1 \primal'_1,\ldots, \primal_d \primal'_d}$
is for the Hadamard (entrywise) product,
for any $\primal$, $\primal'$ in~$ \RR^{d}$.
For any \( \primal \in \RR^d \),
we denote by \( \sign{\primal}\in \{-1,0,1\}^d \)
the vector of~$ \RR^{d}$ with components 
the signs $\sign{\primal_i}\in \na{-1,0,1}$ 
of the entries~$\primal_i$, for $i\in\ic{1,d}$.
Let $\primal\in \RR^d\backslash\{0\}$ be a given vector 
(the case $\primal=0$ is trivial).
We easily obtain that a vector $\dual$ is
\begin{itemize}
\item 
  $\ell_2$-dual to $\primal$ iff (if and only if)
  there exists $\lambda \in \RR_{+}$ such that $\dual=\lambda \primal$;
\item 
  $\ell_p$-dual to $\primal$ for $p\in ]1,\infty[$ iff
  there exists $\lambda \in \RR_{+}$ such that
  $\dual=\lambda\mathrm{sign}(\primal)\circ \bp{|\primal_i|^{p/q}}_{i
    \in\ic{1,d}}$, where $q$ is such that \(1/p + 1/q=1\);
\item 
  $\ell_{1}$-dual to $\primal$ iff 
  the vectors $\dual$ and $\norm{\dual}_{\infty} \mathrm{sign}(\primal)$
  coincide on ${\Support{\primal}}$, the support of the vector~$\primal$ as defined in~\eqref{eq:support_of_a_vector};
\item 
  $\ell_{\infty}$-dual to $\primal$ iff
  \( \dual_j=0 \) for all \( j \in \argmax_{i\in \ic{1,d}} |\primal_i| \), and $\dual\circ\primal \ge 0$.
\end{itemize}

\subsubsection*{Restriction norms}

For any subset \( K \subset \ic{1,d} \), 
we denote by \( \RR^K \) the set of functions from~$K$ to~$\RR$ ---
which can be identified with~\( \RR^{\cardinal{K}} \), where $\cardinal{K}$ denotes the cardinality of 
 \( K \subset \ic{1,d} \)) --- and we introduce 
the subspace of~$\RR^d$ made of vectors
whose components vanish outside of~$K$ by\footnote{%
  Here, following notation from game theory, 
  we have denoted by $-K$ the complementary subset 
  of~$K$ in \( \ic{1,d} \): \( K \cup (-K) = \ic{1,d} \)
  and \( K \cap (-K) = \emptyset \).}
\begin{equation}
  \FlatRR_{K} = \RR^K \times \{0\}^{-K} =
  \bset{ \primal \in \RR^d }{ \primal_j=0 \eqsepv \forall j \not\in K } 
  \subset \RR^d 
  \eqfinv
  \label{eq:FlatRR}
\end{equation}
where \( \FlatRR_{\emptyset}= \na{0} \).
We denote by \( \pi_K : \RR^d \to \FlatRR_{K} \) 
the \emph{orthogonal projection mapping}
and, for any vector \( \primal \in \RR^d \), by 
\( \primal_K = \pi_K\np{\primal} \in \FlatRR_{K} \) 
the vector which coincides with~\( \primal \),
except for the components outside of~$K$ that are zero.
It is easily seen that the orthogonal projection mapping~$\pi_K$ 
is self-dual (equal to its dual operator), giving
\begin{equation}
  \proscal{\primal_K}{\dual_K} 
  = \proscal{\primal_K}{\dual} 
  = \bscal{\pi_K\np{\primal}}{\dual} 
  = \bscal{\primal}{\pi_K\np{\dual}}
  = \proscal{\primal}{\dual_K} 
  \eqsepv 
  \forall \primal \in \RR^d 
  \eqsepv 
  \forall \dual \in \RR^d 
  \eqfinp
  \label{eq:orthogonal_projection_self-dual}
\end{equation}


\begin{definition}
  For any norm~$\TripleNorm{\cdot}$ on~$\RR^d$
  and any subset \( K \subset\ic{1,d} \),
  we define three norms on the subspace~\( \FlatRR_{K} \) of~\( \RR^d \),
  as defined in~\eqref{eq:FlatRR}, as follows.
  \begin{itemize}
  \item 
    The \emph{$K$-restriction norm} \( \TripleNorm{\cdot}_{K} \)
    is the norm on~\( \FlatRR_{K} \) defined by   
    \begin{equation}
      \TripleNorm{\primal}_{K} = \TripleNorm{\primal} 
      \eqsepv
      \forall \primal \in \FlatRR_{K} 
      \eqfinp 
      \label{eq:K_norm}
    \end{equation}
  \item 
    The $\StarSet{K}$-norm
    \( \TripleNorm{\cdot}_{\star,K} \) is 
    the norm \( \bp{\TripleNorm{\cdot}_{\star}}_{K} \),
    given by the restriction to the subspace~\( \FlatRR_{K} \) of
    the dual norm~$\TripleNormDual{\cdot}$  (first dual, as recalled in
    definition~\eqref{eq:dual_norm} of a dual norm, then restriction),
  \item 
    The $\SetStar{K}$-norm
    \( \TripleNorm{\cdot}_{K,\star} \) is 
    the norm \( \bp{\TripleNorm{\cdot}_{K}}_{\star} \),
    given by the dual norm (on the subspace~\( \FlatRR_{K} \))
    of the $K$-restriction norm~\( \TripleNorm{\cdot}_{K} \) 
    to the subspace~\( \FlatRR_{K} \) (first restriction, then dual).
  \end{itemize}
  \label{de:K_norm}
\end{definition}

\bleue{%
It has been established 
(see \cite[Proposition~2.2]{Marques_de_Sa-Sodupe:1993})
that, for any nonempty subset \( K \subset \ic{1,d} \), 
  one has the inequality \( \TripleNorm{\cdot}_{K,\star} 
  \leq \TripleNorm{\cdot}_{\star,K} \).
  We will discuss the equality case in Proposition~\ref{pr:orthant-monotonic}. 
}

\subsection{New characterizations of orthant-monotonic norms}
\label{Orthant-monotonic_norms}

We recall the definitions of monotonic and of orthant-monotonic norms
before introducing, in the next \S\ref{Orthant-strictly_monotonic_norms},
the new notion of orthant-strictly monotonic norms.
For any \( \primal=\np{\primal_1,\ldots,\primal_d} \in \RR^d \), we denote \( \module{\primal}
=\np{\module{\primal_1},\ldots,\module{\primal_d}} \in \RR^d \). 

\begin{definition}
  \label{de:orthant-monotonic}
  A norm \( \TripleNorm{\cdot}\) on the space~\( \RR^d \) is called
  \begin{itemize}
  \item 
    \emph{monotonic}~\cite{Bauer-Stoer-Witzgall:1961}
    if, for all 
    $\primal$, $\primal'$ in~$ \RR^{d}$, we have 
    \(
    |\primal| \le |\primal'| \Rightarrow 
    \TripleNorm{\primal} \le \TripleNorm{\primal'}
    \),
    where $|\primal| \leq |\primal'|$ means 
    $|\primal_i| \leq |\primal^{'}_i|$ for all $i\in\ic{1,d}$, 
  \item 
    \emph{orthant-monotonic}~\cite{Gries:1967,Gries-Stoer:1967}
    if, for all 
    $\primal$, $\primal'$ in~$ \RR^{d}$, we have 
    \bp{  \(
      |\primal| \le |\primal'| \text{ and } 
      \primal\circ\primal' \ge 0 \Rightarrow 
      \TripleNorm{\primal} \le \TripleNorm{\primal'}
      \) }.
  \end{itemize}
\end{definition}

We will use the following, easy to prove, properties: 
any monotonic norm is orthant-monotonic;
if a norm is orthant-monotonic, so are its restriction norms
in Definition~\ref{de:K_norm}
(as norms on their respective subspaces).
All the $\ell_p$-norms~$\norm{\cdot}_{p}$, 
for $p\in [1,\infty]$, are monotonic, hence orthant-monotonic.
The definition of an \emph{orthant-monotonic seminorm} is straightforward,
and it is easily proven that
the supremum of a family of orthant-monotonic seminorms is 
an orthant-monotonic seminorm. 
\medskip

We recall the definitions of Birkhoff orthogonality~\cite{Birkhoff:1935},
and of strict Birkhoff orthogonality~\cite{Sain-Paul-Jha:2015}. 

\begin{definition}\label{de:Birkhoff_orthogonal} 
  Let ${\cal U}$ and ${\cal V}$ be two subspaces of~$\RR^d$.
  Let $\TripleNorm{\cdot}$ be a norm on~$\RR^d$.
  \begin{itemize}
  \item 
    We say that the subspace~${\cal U}$ is \emph{Birkhoff orthogonal}~\cite{Birkhoff:1935}
    to the subspace~${\cal V}$, 
    denoted by \( {\cal U} \perp_{\TripleNorm{\cdot}} {\cal V} \) if 
    \(       \TripleNorm{u + v} \geq \TripleNorm{u} \), 
    for any \( u \in {\cal U} \) and any \( v \in {\cal V} \), that is, 
    \begin{equation}
      {\cal U} \perp_{\TripleNorm{\cdot}} {\cal V} \iff
      \TripleNorm{u + v} \geq \TripleNorm{u}
      \eqsepv \forall u \in {\cal U}
      \eqsepv \forall v \in {\cal V}
      \eqfinp
      \label{eq:Birkhoff_orthogonal} 
    \end{equation}
  \item 
    We say that the subspace~${\cal U}$ is \emph{strictly Birkhoff orthogonal}~\cite{Sain-Paul-Jha:2015} 
    to the subspace~${\cal V}$, 
    denoted by \( {\cal U} \perp_{\TripleNorm{\cdot}}^{>} {\cal V} \) if 
    \(       \TripleNorm{u + v} > \TripleNorm{u} \), 
    for any \( u \in {\cal U} \) and any \( v \in {\cal V}\backslash\{0\} \), that is,
    \begin{equation}
      {\cal U} \perp_{\TripleNorm{\cdot}}^{>} {\cal V} \iff 
      \TripleNorm{u + v} > \TripleNorm{u}
      \eqsepv \forall u \in {\cal U}
      \eqsepv \forall v \in {\cal V}\backslash\{0\}
      \eqfinp
      \label{eq:strictly_Birkhoff_orthogonal} 
    \end{equation}
  \end{itemize}
\end{definition}

Now, we are ready to recall established characterizations 
of orthant-monotonic norms,
and to add new characterizations, namely Item~\ref{it:ICS} and Item~\ref{it:projection}
in the following Proposition~\ref{pr:orthant-monotonic}.

\begin{proposition}
  Let $\TripleNorm{\cdot}$ be a norm on~$\RR^d$.
  The following assertions are equivalent.
  \begin{enumerate}
  \item 
    The norm $\TripleNorm{\cdot}$ is orthant-monotonic.
    \label{it:orthant-monotonic_IN_pr:orthant-monotonic}
  \item 
    The norm $\TripleNormDual{\cdot}$ is orthant-monotonic.
    \label{it:dual_orthant-monotonic_IN_pr:orthant-monotonic}
  \item 
    \( \TripleNorm{\cdot}_{K,\star} = \TripleNorm{\cdot}_{\star,K} \),
    for all \( K \subset\ic{1,d} \).
    \label{it:K_star_=_star_K_IN_pr:orthant-monotonic}
  \item 
    \( \FlatRR_{K} \perp_{\TripleNorm{\cdot}} \FlatRR_{-K} \),
    for all \( K \subset\ic{1,d} \).
    \label{it:Birkhoff_orthogonal_IN_pr:orthant-monotonic}
  \item 
    \( \FlatRR_{K} \perp_{\TripleNorm{\cdot}} \FlatRR_{-K} \),
    for all \( K \subset\ic{1,d} \) with  $|K|= d-1$.
  \item 
    \label{it:dual_couple_IN_pr:orthant-monotonic}
    For any vector \( u \in \RR^d\backslash\{0\} \),
    there exists a vector \( v \in \RR^d\backslash\{0\} \)
    such that \( \Support{v} \subset \Support{u} \),
    that \( u~\circ~v \geq 0 \) 
    and that $v$ is \( \TripleNorm{\cdot} \)-dual to~$u$
    as in~\eqref{eq:couple_TripleNorm-dual}.
  \item 
    \label{it:ICS}
    The norm $\TripleNorm{\cdot}$ is \emph{increasing with the coordinate subspaces},
    in the sense that, for any \( \primal \in \RR^d \)
    and any \( J \subset K \subset\ic{1,d} \),
    we have $ \TripleNorm{\primal_{J}} \leq \TripleNorm{\primal_K}$.
  \item 
    \( \pi_K\np{\TripleNormBall} = \FlatRR_{K} \cap \TripleNormBall \),
    for all \( K \subset\ic{1,d} \).
    \label{it:projection}
  \end{enumerate}
  \label{pr:orthant-monotonic}
\end{proposition}

\begin{proof}
  The equivalence between all statements but the two last ones can be found
  in~\cite[Proposition~2.4]{Marques_de_Sa-Sodupe:1993}.

  It is easily established that Item~\ref{it:ICS} is equivalent 
  to Item~\ref{it:Birkhoff_orthogonal_IN_pr:orthant-monotonic}.
  Indeed, suppose that Item~\ref{it:ICS} holds true. 
  We consider \( \primal \in \RR^d \)
  and \( J \subset K \subset\ic{1,d} \).
  By setting \( u=\primal_{J} \in \FlatRR_{J} \)
  and \( v = \primal_K-\primal_{J} \), we get that 
  \( v \in \FlatRR_{-J} \). By Item~\ref{it:ICS}, we have that 
  \( \TripleNorm{u} \leq \TripleNorm{u + v} \),
  hence that $ \TripleNorm{\primal_{J}} \leq \TripleNorm{\primal_K}$.
  The reverse implication is proved in the same way. 
  
  We now show that
  Item~\ref{it:K_star_=_star_K_IN_pr:orthant-monotonic}
  and Item~\ref{it:projection} are equivalent.
  For this purpose, let $\TripleNorm{\cdot}$ be a norm on~$\RR^d$
  and \( K \subset\ic{1,d} \), and let us admit for a while that
  \begin{subequations}
    \begin{align}
      \TripleNorm{\dual}_{\star,K} 
      &=
        \sigma_{\pi_K\np{\TripleNormBall}}\np{\dual} 
        = \sigma_{\pi_K\np{\TripleNormSphere}}\np{\dual} 
        \eqsepv \forall \dual \in \FlatRR_{K} 
        \eqfinv
        \label{eq:star_K=sigma}
      \\
      \TripleNorm{\dual}_{K,\star}
      &=
        \sigma_{ \FlatRR_{K} \cap \TripleNormBall }\np{\dual}  
        = \sigma_{ \FlatRR_{K} \cap \TripleNormSphere }\np{\dual}
        \eqsepv \forall \dual \in \FlatRR_{K} 
        \eqfinp
        \label{eq:K_star=sigma}
    \end{align}
  \end{subequations}
  Therefore, the equality  
  \( \TripleNorm{\cdot}_{\star,K} = 
  \TripleNorm{\cdot}_{K,\star} \) is equivalent to
  \( \sigma_{\pi_K\np{\TripleNormBall}} =
  \sigma_{ \FlatRR_{K} \cap \TripleNormBall } \),
  when this last equality is restricted to the subspace~$\FlatRR_{K}$.
  Now, on the one hand,
  the subset \( \pi_K\np{\TripleNormBall} \) of~$\FlatRR_{K}$
  is convex and closed (in the subspace~$\FlatRR_{K}$) 
  as the image of the convex and compact set~$\TripleNormBall$ 
  by the linear mapping~$\pi_K$.
  On the other hand, 
  the subset \( \FlatRR_{K} \cap \TripleNormBall \) of~$\FlatRR_{K}$
  is convex and closed (in the subspace~$\FlatRR_{K}$).
  Therefore, \( \TripleNorm{\cdot}_{\star,K} = 
  \TripleNorm{\cdot}_{K,\star} \) if and only if 
  \( \pi_K\np{\TripleNormBall} =
  \FlatRR_{K} \cap \TripleNormBall \).
  Thus, we have shown that Item~\ref{it:K_star_=_star_K_IN_pr:orthant-monotonic}
  and Item~\ref{it:projection} are equivalent.
  It remains to prove~\eqref{eq:star_K=sigma} and \eqref{eq:K_star=sigma}. 
  \medskip
  
  \noindent $\bullet$
  We prove~\eqref{eq:star_K=sigma}.
  For any \( \dual \in \FlatRR_{K} \), we have
  \begin{align*}
    \TripleNorm{\dual}_{\star,K} 
    &= 
      \TripleNorm{\dual}_{\star}
      \tag{using Definition~\ref{de:K_norm}}
    \\
    &= 
      \sigma_{\TripleNormBall}\np{\dual}
      \tag{by~\eqref{eq:norm_dual_norm} }
    \\
    &= 
      \sup_{\primal \in \TripleNormBall} \proscal{\primal}{\dual} 
      \tag{by definition~\eqref{eq:support_function} of the support function~$\sigma_{\TripleNormBall}$ }
    \\
    &= 
      \sup_{\primal \in \TripleNormBall} \proscal{\primal}{\pi_K\np{\dual}}
      \tag{as \( \dual=\pi_K\np{\dual} \) because \( \dual \in \FlatRR_{K} \) }
    \\
    &= 
      \sup_{\primal \in \TripleNormBall} \proscal{\pi_K\np{\primal}}{\dual}
      \tag{by the self-duality property~\eqref{eq:orthogonal_projection_self-dual} of 
      the projection mapping~\( \pi_K \) }
    \\
    &= 
      \sup_{\primal' \in \pi_K\np{\TripleNormBall}} \proscal{\primal'}{\dual}
    \\
    &= 
      \sigma_{\pi_K\np{\TripleNormBall}}\np{\dual} 
      \tag{by definition~\eqref{eq:support_function} of the support function~$\sigma_{\pi_K\np{\TripleNormBall}}$
      }
      \eqfinp 
  \end{align*}
  Thus, we have proved that  \( \TripleNorm{\dual}_{\star,K} = 
  \sigma_{\pi_K\np{\TripleNormBall}}\np{\dual} \).

  It remains to prove that \( \sigma_{\pi_K\np{\TripleNormBall}}\np{\dual} 
  = \sigma_{\pi_K\np{\TripleNormSphere}}\np{\dual} \).
  Now, as the unit ball \( \TripleNormBall \) is equal to the convex hull
  \( \convexhull\np{\TripleNormSphere} \) 
  of the unit sphere~\( \TripleNormSphere \),
  we get that \( \pi_K\np{\TripleNormBall} = 
  \pi_K\np{ \convexhull\np{\TripleNormSphere} } \).
  As $\pi_K$ is a linear mapping, we easily obtain that
  \( \pi_K\np{ \convexhull\np{\TripleNormSphere} } = 
  \convexhull\np{\pi_K\np{\TripleNormSphere}} \). 
  Since \( \sigma_{ \convexhull\np{\pi_K\np{\TripleNormSphere}} } =
  \sigma_{\pi_K\np{\TripleNormSphere}} \)
  \cite[Prop.~7.13]{Bauschke-Combettes:2017}, 
  we conclude that 
  \( \TripleNorm{\dual}_{\star,K} = 
  \sigma_{\pi_K\np{\TripleNormBall}}
  = \sigma_{ \convexhull\np{\pi_K\np{\TripleNormSphere}} }
  = \sigma_{\pi_K\np{\TripleNormSphere}} \) on~\( \FlatRR_{K} \),
  that is, equality~\eqref{eq:star_K=sigma} holds true. 
  \medskip

  \noindent $\bullet$
  We prove~\eqref{eq:K_star=sigma}. 
  
  By~\eqref{eq:norm_dual_norm}, we have the equality 
  \( \TripleNorm{\cdot}_{K,\star} = \sigma_{  \FlatRR_{K} \cap \TripleNormBall } \) 
  on~$\FlatRR_{K} $, as \( \FlatRR_{K} \cap \TripleNormBall \) is easily
  seen to be the unit ball (in~$\FlatRR_{K} $) 
  of the restriction norm~\( \TripleNorm{\cdot}_{K} \) in~\eqref{eq:K_norm}.
  Therefore, we have proved that 
  \( \TripleNorm{\dual}_{K,\star}=
  \sigma_{ \FlatRR_{K} \cap \TripleNormBall }\np{\dual} \)
  for any \( \dual \in \FlatRR_{K} \).

  Now, we prove that 
  \( \sigma_{ \FlatRR_{K} \cap \TripleNormBall }\np{\dual} 
  = \sigma_{ \FlatRR_{K} \cap \TripleNormSphere }\np{\dual} \)
  for any \( \dual \in \FlatRR_{K} \).
  It is easy to check that the unit sphere (in~$\FlatRR_{K} $) 
  of the restriction norm~\( \TripleNorm{\cdot}_{K} \) in~\eqref{eq:K_norm}
  is $\FlatRR_{K} \cap \TripleNormSphere$. 
  Then, using the fact that the convex hull 
  (be it in~$\FlatRR_{K} $ or in~$\RR^d$) 
  of the unit sphere $\FlatRR_{K} \cap \TripleNormSphere$
  is the unit ball \( \FlatRR_{K} \cap \TripleNormBall \),
  we have that \( \convexhull\np{ \FlatRR_{K} \cap \TripleNormSphere } 
  = \FlatRR_{K} \cap \TripleNormBall \).
  As \( \sigma_{ \convexhull\np{ \FlatRR_{K} \cap \TripleNormSphere } } = 
  \sigma_{ \FlatRR_{K} \cap \TripleNormSphere } \)
  \cite[Prop.~7.13]{Bauschke-Combettes:2017}, 
  we conclude that 
  \( \TripleNorm{\cdot}_{K,\star}
  = \sigma_{ \FlatRR_{K} \cap \TripleNormBall }
  = \sigma_{ \convexhull\np{ \FlatRR_{K} \cap \TripleNormSphere } } 
  = \sigma_{ \FlatRR_{K} \cap \TripleNormSphere } \) on~\( \FlatRR_{K} \),
  that is, equality~\eqref{eq:K_star=sigma} holds true. 
\end{proof}

As an example, we illustrate Item~\ref{it:dual_couple_IN_pr:orthant-monotonic}
of Proposition~\ref{pr:orthant-monotonic} with the
$\ell_1$ and $\ell_{\infty}$ norms, which both are orthant-monotonic.
Let ${\mathbb I}\in \RR^d$ denote the vector whose components are all equal to one.
For any vector \( u \in \RR^d \), 
\begin{itemize}
\item 
  the vector
  $v = \mathrm{sign}(u)$ is such that
  \( \Support{v} = \Support{u} \),
  that \( u~\circ~v \geq 0 \),
  and is $\norm{\cdot}_{1}$-dual to the vector~$u$;  
  this last assertion is obvious for $u=0$ and, when
  $u\not=0$, we have that
  \[
    \proscal{u}{v} {=} \proscal{u}{\mathrm{sign}(u)}
    = \proscal{|u|}{{\mathbb I}}
    = \norm{u}_1   = \norm{u}_1 \times 1 
    = \norm{u}_1\norm{v}_\infty
    \eqfinv
  \]
\item 
  the vector
  $v=\mathrm{sign}(u) \circ {\mathbb I}_{U}$, 
  where $U=\argmax_{i\in \ic{1,d}} |u_i|$,
  is such that
  \( \Support{v} \subset \Support{u} \),
  that \( u~\circ~v \geq 0 \),
  and is  $\norm{\cdot}_{\infty}$-dual to the vector~$u$, as we have
  \[
    \proscal{u}{v} = \proscal{u}{\mathrm{sign}(u)~\circ~{\mathbb I}_{U}} 
    = \proscal{|u|_U}{{\mathbb I}_{U}} 
    = \bscal{ \norm{u}_\infty {\mathbb I}_{U}}{{\mathbb I}_{U}}
    = \norm{u}_\infty \norm{{\mathbb I}_{U}}_1
    = \norm{u}_\infty \norm{v}_1
    \eqfinp
  \]
\end{itemize}


\subsection{Orthant-strictly monotonic norms}
\label{Orthant-strictly_monotonic_norms}

After these recalls, 
we introduce two new notions, that are the strict versions of monotonic
and orthant-monotonic norms.
Then, we provide characterizations 
that will prove especially relevant for the \lzeropseudonorm.

\begin{definition}  
  \label{de:orthant-strictly_monotonic}
  A norm \( \TripleNorm{\cdot}\) on the space~\( \RR^d \) is called
  \begin{itemize}
  \item 
    \emph{strictly monotonic} if, for all 
    $\primal$, $\primal'$ in~$ \RR^{d}$, we have 
    \(
    |\primal| < |\primal'| \Rightarrow 
    \TripleNorm{\primal} < \TripleNorm{\primal'}
    \),
    where \( |\primal| < |\primal'| \) means that 
    $|\primal_i| \le |\primal^{'}_i|$ for all $i\in\ic{1,d}$, 
    and that there exists $j \in \ic{1,d}$ such that
    $|\primal_j| < |\primal^{'}_j|$,
  \item 
    \emph{orthant-strictly monotonic} if, for all 
    $\primal$, $\primal'$ in~$ \RR^{d}$, we have 
    \bp{  \(
      |\primal| < |\primal'| \text{ and } 
      \primal~\circ~\primal' \ge 0 \Rightarrow 
      \TripleNorm{\primal} < \TripleNorm{\primal'}
      \) }.
  \end{itemize}
\end{definition}
We will use the following, easy to prove, properties: 
any strictly monotonic norm is orthant-strictly monotonic;
any orthant-strictly monotonic norm is orthant-monotonic.

All the $\ell_p$-norms~$\norm{\cdot}_{p}$ on the space~\( \RR^d \), 
for $p\in [1,\infty[$, are strictly monotonic, 
hence orthant-strictly monotonic. By contrast, 
the $\ell_\infty$-norm~$\norm{\cdot}_{\infty}$ is 
not orthant-strictly monotonic.

To the difference with orthant-monotonicity (equivalence between
Item~\ref{it:orthant-monotonic_IN_pr:orthant-monotonic}
and
Item~\ref{it:dual_orthant-monotonic_IN_pr:orthant-monotonic}
of Proposition~\ref{pr:orthant-monotonic}),
the notion of orthant-strictly monotonicity is not 
necessarily preserved when taking the dual norm:
indeed, the $\ell_1$-norm~$\norm{\cdot}_{1}$ is orthant-strictly monotonic,
whereas its dual norm, the $\ell_\infty$-norm~$\norm{\cdot}_{\infty}$ 
is orthant-monotonic, but not orthant-strictly monotonic.
\medskip

Now, we provide characterizations of orthant-strictly monotonic norms.

\begin{proposition}
  Let $\TripleNorm{\cdot}$ be a norm on~$\RR^d$.
  The following assertions are equivalent.
  \begin{enumerate}
  \item
    \label{it:OSM}
    The norm $\TripleNorm{\cdot}$ is orthant-strictly monotonic.
  \item
    \label{it:SBO}
    The family 
    \( \sequence{ \FlatRR_{K} }{K \subset \ic{1,d}} \)
    of subspaces of~$\RR^d$ is \emph{strictly Birkhoff orthogonal}, 
    in the sense that \( \FlatRR_{K} \perp_{\TripleNorm{\cdot}}^{>} \FlatRR_{-K} \),
    for all \( K \subset\ic{1,d} \), as in~\eqref{eq:strictly_Birkhoff_orthogonal}.
  \item
    \label{it:SICS}
    The norm $\TripleNorm{\cdot}$ is \emph{strictly increasing
      with the coordinate subspaces},
    in the sense that\footnote{%
      By \( J \subsetneq K \), we mean that  \( J \subset K \) and \( J \neq K \). },
    for any \( \primal \in \RR^d \)
    and any \( J \subsetneq K \subset\ic{1,d} \),
    we have $ \primal_J \neq \primal_K \Rightarrow
    \TripleNorm{\primal_{J}} < \TripleNorm{\primal_K}$.
  \item
    \label{it:SDC}
    For any vector \( u \in \RR^d\backslash\{0\} \), 
    there exists a vector \( v \in \RR^d\backslash\{0\} \)
    such that \( \Support{v} = \Support{u} \),
    that \( u~\circ~v \geq 0 \), and that 
    $v$ is \( \TripleNorm{\cdot} \)-dual to~$u$, that is,
    \( \proscal{u}{v} = \TripleNorm{u} \times \TripleNormDual{v} \). 
  \end{enumerate}
  \label{pr:orthant-strictly_monotonic}
\end{proposition}

\begin{proof}

  \noindent $\bullet$ 
  We prove that Item~\ref{it:OSM} implies Item~\ref{it:SBO}.

  Let  \( K \subset\ic{1,d} \). 
  Let \( u \in \FlatRR_{K} \) and \( v \in \FlatRR_{-K}\backslash\{0\} \),
  that is, \( u=u_{K} \) and \( v=v_{-K}  \neq 0 \).
  We want to show that \( \TripleNorm{u + v} > \TripleNorm{u} \),
  by the definition~\eqref{eq:strictly_Birkhoff_orthogonal} 
  of strict Birkhoff orthogonality. 

  On the one hand, 
  by definition of the module of a vector, 
  we easily see that 
  \( \module{\primal} = \module{ \primal_{K} } + \module{ \primal_{-K} } \),
  for any vector~\( \primal \in \RR^d \). 
  Thus, we have
  \( \module{u+v} = \module{ \np{u+v}_{K} } + \module{ \np{u+v}_{-K} } 
  = \module{ u_{K}+v_{K} } + \module{ u_{-K}+v_{-K} } 
  = \module{ u_{K}+0 } + \module{ 0+v_{-K} } 
  = \module{ u_{K} } + \module{v_{-K} } > \module{ u_{K} } =\module{ u } \) 
  since \( \module{v_{-K} } >0 \) as \( v=v_{-K} \neq 0 \),
  and since \( u=u_{K} \).  
  On the other hand, we easily get that \( \np{u+v}~\circ~u = 
  \bp{ \np{u+v}_{K}~\circ~u_{K} } + \bp{ \np{u+v}_{-K}~\circ~u_{-K} } = 
  \bp{ u_{K}~\circ~u_{K} } + \bp{ v_{-K}~\circ~u_{-K} } = 
  \bp{ u_{K}~\circ~u_{K} } \), because $u_{-K}=0$ and $v_{K}=0$.
  Therefore, we get that 
  \( \np{u+v}~\circ~u = \bp{ u_{K}~\circ~u_{K} } \geq 0 \).

  From \( \module{u+v} > \module{u} \) and 
  \( \np{u+v}~\circ~u \geq 0 \),
  we deduce that \( \TripleNorm{u + v} > \TripleNorm{u} \)
  by Definition~\ref{de:orthant-strictly_monotonic}
  as the norm \( \TripleNorm{\cdot} \) is orthant-strictly monotonic.
  Thus, \eqref{eq:strictly_Birkhoff_orthogonal} is satisfied,
  hence Item~\ref{it:SBO} holds true.
  \medskip

  \noindent $\bullet$ 
  We prove that Item~\ref{it:SBO} implies Item~\ref{it:SICS}.

  Let \( \primal \in \RR^d \)
  and \( J \subsetneq K \subset\ic{1,d} \)
  be and such that $ \primal_J \neq \primal_K $.
  We will show that \( \TripleNorm{\primal_{J}} < \TripleNorm{\primal_K} \).

  As \( J \subsetneq K \subset\ic{1,d} \)
  and $ \primal_J \neq \primal_K $, 
  there exists \( w \in \FlatRR_{-J} \), \( w \neq 0 \), 
  such that \( \primal_K=\primal_J+w \).
  Now, as the family 
  \( \sequence{ \FlatRR_{K} }{K \subset \ic{1,d}} \)
  is strictly Birkhoff orthogonal by assumption (Item~\ref{it:SBO}),
  we have \( \FlatRR_{J} \perp_{\TripleNorm{\cdot}}^{>} \FlatRR_{-J} \).
  As a consequence, we obtain that 
  \( \TripleNorm{\primal_{K}} =  \TripleNorm{\primal_J+w} > \TripleNorm{\primal_J} \).
  \medskip

  \noindent $\bullet$ 
  We prove that Item~\ref{it:SICS} implies Item~\ref{it:SDC}. 

  Let \( u \in \RR^d\backslash\{0\} \) be given 
  and let us put \( K = \Support{u} \neq \emptyset \).
  As the norm \( \TripleNorm{\cdot} \) is orthant-strictly monotonic,
  it is orthant-monotonic; hence, 
  by Item~\ref{it:dual_couple_IN_pr:orthant-monotonic}
  in Proposition~\ref{pr:orthant-monotonic},
  there exists a vector \( v \in \RR^d\backslash\{0\} \)
  such that \( \Support{v} \subset \Support{u} \),
  that \( u~\circ~v \geq 0 \) 
  and that $v$ is \( \TripleNorm{\cdot} \)-dual to~$u$,
  as in~\eqref{eq:couple_TripleNorm-dual},
  that is,  \( \proscal{u}{v} = \TripleNorm{u} \times \TripleNormDual{v} \). 
  Thus \( J = \Support{v} \subset K = \Support{u} \).
  We now show that \( J \subsetneq K  \) is impossible,
  hence that \( J = K  \), thus proving that
  Item~\ref{it:SDC} holds true with the above vector~$v$.

  Writing that  \( \proscal{u}{v} = \TripleNorm{u} \times \TripleNormDual{v} \)
  (using that $u=u_K$ and $v=v_K=v_J$), we obtain
  \[
    \TripleNorm{u} \times \TripleNormDual{v} =
    \proscal{u}{v} = \proscal{u_K}{v} = \proscal{u_K}{v_K} = \proscal{u_K}{v_J} = 
    \proscal{u_J}{v_J} = \proscal{u_J}{v} 
    \eqfinp
  \]
  As a consequence, 
  \( \{u_K,u_J \} \subset \argmax_{ \TripleNorm{\primal} \leq \TripleNorm{u} }
  \proscal{\primal}{v} \), by definition~\eqref{eq:dual_norm} of \(
  \TripleNormDual{v} \),
  because \( \TripleNorm{u}=\TripleNorm{u_K} \geq \TripleNorm{u_J} \),
  by Item~\ref{it:ICS} in Proposition~\ref{pr:orthant-monotonic}
  since \( J \subset K \) and the norm \( \TripleNorm{\cdot} \) 
  is orthant-monotonic. 
  But any solution in \( \argmax_{ \TripleNorm{\primal} \leq \TripleNorm{u} }
  \proscal{\primal}{v} \) belongs to the frontier of the ball of
  radius~$\TripleNorm{u}$,
  hence has exactly norm~$\TripleNorm{u}$. 
  Thus, we deduce that \( \TripleNorm{u}=\TripleNorm{u_K}= \TripleNorm{u_J} \).
  If we had \( J = \Support{v} \subsetneq K = \Support{u} \), 
  we would have $u_J \neq u_K$, hence 
  \( \TripleNorm{u_K} > \TripleNorm{u_J} \) by Item~\ref{it:SICS};
  this would be in contradiction with \( \TripleNorm{u_K} = \TripleNorm{u_J} \).
  Therefore,  \( J = \Support{v} =K= \Support{u} \). 
  \medskip

  \noindent $\bullet$ 
  We prove that Item~\ref{it:SDC} implies Item~\ref{it:OSM}.

  Let $\primal$, $\primal'$ in~$ \RR^{d}$ be such that 
  \( |\primal| < |\primal'| \) and 
  \( \primal~\circ~\primal' \ge 0 \).
  We are going to prove that \( \TripleNorm{\primal} < \TripleNorm{\primal'} \).

  We suppose that $\primal\neq 0 $ (otherwise the proof is trivial).
  By Item~\ref{it:SDC}, there exists
  a vector \( w \in \RR^d \) such that \( \Support{w} = \Support{\primal} \),
  \( \primal~\circ~w \ge 0 \) 
  and that \( \proscal{\primal}{w} = \TripleNorm{\primal} \times
  \TripleNormDual{w} \).
  As  \( \Support{w} = \Support{\primal} \) with $\primal\neq 0 $,
  we have \( w\neq 0 \), so that 
  we can always suppose that $\TripleNormDual{w}=1$ (after renormalization), giving 
  \( \TripleNorm{\primal} = \proscal{\primal}{w} \).

  First, we are going to establish that \(  i \in \Support{\primal} \Rightarrow 
  \primal'_i w_i \geq \primal_i w_i  \). 
  From \( |\primal'| > |\primal| \), we deduce that 
  \( |\primal'|^2 \geq |\primal'|~\circ~|\primal| \), and, 
  as \( \primal'~\circ~\primal \ge 0 \),
  we obtain that \( |\primal'|^2 \geq \primal'~\circ~\primal =
  |\primal'|~\circ~|\primal| \ge 0 \). 
  Hence, we deduce
  \[
    \np{ \primal'~\circ~\primal }~\circ~\np{ \primal'~\circ~w } 
    = |\primal'|^2~\circ~\np{ \primal~\circ~w } 
    \geq \np{ \primal'~\circ~\primal }~\circ~\np{ \primal~\circ~w } 
    \eqfinv
  \]
  as \( \primal~\circ~w \ge 0 \). 
  Moving to components, we get that, for all $i\in\ic{1,d}$, 
  \( \primal'_i \primal_i \primal'_i w_i \geq \primal'_i \primal_i \primal_i w_i
  \), so that, on the one hand
  \[
    \primal'_i \primal_i > 0 \Rightarrow   \primal'_i w_i \geq \primal_i w_i
    \eqfinp
  \]
  On the other hand, as \( |\primal'| > |\primal| \) and \( \primal~\circ~\primal' \ge 0 \),
  we easily get that \( \primal'_i \primal_i > 0 \iff i \in \Support{\primal}
  \).
  Therefore, we deduce that \(  i \in \Support{\primal} \Rightarrow \primal'_i \primal_i > 0 
  \Rightarrow \primal'_i w_i \geq \primal_i w_i  \). 

  Second, we show that \( \TripleNorm{\primal} \leq \TripleNorm{\primal'} \).
  Indeed, we have:
  \begin{align*}
    \TripleNorm{\primal'} 
    &=
      \sup_{ \TripleNormDual{w'} \leq 1 } \proscal{\primal'}{w'} 
      \tag{by~\eqref{eq:dual_norm} as
      $\TripleNorm{\cdot}=\np{\TripleNormDual{\cdot}}_\star$ } 
    \\
    & \geq
      \proscal{\primal'}{w} 
      \tag{as $\TripleNormDual{w}=1$ } 
    \\
    &=
      \sum_{ i \in \Support{w} } \primal'_i w_i
    \\
    &= 
      \sum_{ i \in \Support{\primal} } \primal'_i w_i 
      \tag{as \( \Support{w} = \Support{\primal} \) } 
    \\
    &\geq 
      \sum_{ i \in \Support{\primal} } \primal_i w_i 
      \tag{as \( i \in \Support{\primal} \Rightarrow \primal'_i w_i \geq \primal_i w_i  \) }
    \\
    &= 
      \proscal{\primal}{w} 
    \\
    &= \TripleNorm{\primal} 
      \tag{by the property \( \TripleNorm{\primal} = \proscal{\primal}{w} \) of the vector~$w$. }
  \end{align*}

  Third, we show that \( \TripleNorm{\primal} < \TripleNorm{\primal'} \).
  There are two cases.

  In the first case, there exists \( j \in \Support{\primal} \) such that
  \( 0< |\primal_j| < |\primal'_j| \). As a consequence, on the one hand, 
  \( 0< |w_j| |\primal_j| < |w_j| |\primal'_j| \), since $w_j \neq 0$ because
  \( j \in \Support{\primal} = \Support{w} \).
  On the other hand,  \( \primal'_j \primal_j > 0 \) implies 
  \(\primal'_j w_j \geq \primal_j w_j \), as seen above,
  and \( \primal_j w_j \geq 0 \) because \( \primal~\circ~w \ge 0 \).
  Thus, we get that \(\primal'_j w_j \geq \primal_j w_j \geq 0 \). 
  As \( 0< |\primal_j| < |\primal'_j| \), we deduce that \(\primal'_j w_j > \primal_j w_j \).
  Returning to the last inequality in the sequence of 
  equalities and inequalities above, we observe that it is now strict,
  and we conclude that 
  \( \TripleNorm{\primal'} > \TripleNorm{\primal} \).

  In the second case, \( i \in \Support{\primal}  \Rightarrow 
  0< |\primal_i| = |\primal'_i| \).
  As \( |\primal| < |\primal'| \), we deduce that there exists 
  \( j \in \Support{\primal'}\backslash\Support{\primal} \) such that
  \( 0 = |\primal_j| < |\primal'_j| \). 
  We define a new vector \( \tilde{\primal} \) by 
  \( \tilde{\primal}_j=1/2 \primal'_j \neq 0 \) and
  \( \tilde{\primal}_i= \primal_i \) for $i \neq j$.
  Putting \( I=\Support{\primal} \), we have
  \( \tilde{\primal} = \primal_I + 1/2 \primal'_j e_j 
  = \tilde{\primal}_I + \tilde{\primal}_{\{j\}} \),
  where $e_j$ denotes the $j$-canonical vector of~$\RR^d$.
  On the one hand, from the first case we obtain that 
  \( \TripleNorm{\tilde{\primal}} < \TripleNorm{\primal'} \).
  On the other hand, we have 
  \( \TripleNorm{\primal} \leq \TripleNorm{\tilde{\primal}} \);
  indeed, by Proposition~\ref{pr:orthant-monotonic}, 
  Item~\ref{it:SDC} implies that 
  the norm $\TripleNorm{\cdot}$ is orthant-monotonic,
  hence that \( \TripleNorm{\tilde{\primal}} = 
  \TripleNorm{ \tilde{\primal}_I + \tilde{\primal}_{\{j\}} } \geq
  \TripleNorm{ \tilde{\primal}_I } = \TripleNorm{\primal} \).
  We conclude that 
  \( \TripleNorm{\primal} \leq \TripleNorm{\tilde{\primal}} < \TripleNorm{\primal'} \).
  \medskip

  This ends the proof.
\end{proof}

As an example, we illustrate Item~\ref{it:SDC}
of Proposition~\ref{pr:orthant-strictly_monotonic} with the
$\ell_1$ (orthant-strictly monotonic) and $\ell_{\infty}$ (not orthant-strictly monotonic) norms. 
\begin{itemize}
\item
  For any vector \( u \in \RR^d \), 
  we have seen (right after the proof of Proposition~\ref{pr:orthant-strictly_monotonic}) that the vector
  $v = \mathrm{sign}(u)$ is such that
  \( \Support{v} = \Support{u} \),
  that \( u~\circ~v \geq 0 \),
  and is $\norm{\cdot}_1$-dual to the vector~$u$.
  This is another proof that the norm~$\ell_1$ is orthant-strictly monotonic.
\item 
  By contrast, if the vector~$v \neq 0$ is $\norm{\cdot}_{\infty}$-dual to the vector
  $u=\np{1,1/2,0,\ldots,0}$, then an easy computation shows that, necessarily,
  $v=\np{v_1,0,0,\ldots,0}$ with $v_1 > 0$. As a consequence, this gives
  \( \{1\}=\Support{v} \subsetneq \Support{u}=\{1,2\} \).
  This suffices to prove that the norm~$\ell_{\infty}$ is not orthant-strictly monotonic.
\end{itemize}
\medskip

We end this~\S\ref{Orthant-strictly_monotonic_norms}
with additional properties related to exposed and extreme points of 
the unit ball~$\TripleNormBall$ of an orthant-strictly monotonic 
norm~$\TripleNorm{\cdot}$.
We recall that an element $\primal$ of a convex set~$\Convex$ is called an 
\emph{exposed point} of~$\Convex$ if there exists a support hyperplane~$H$ 
to the convex set~$\Convex$ at~$\primal$ such that 
$H\cap \Convex =\na{\primal}$.
We show in the next proposition that orthant-strictly monotonicity implies that
the intersection of the unit sphere~$\TripleNormSphere$ with the subspaces
$\FlatRR_{\na{i}}$ in~\eqref{eq:FlatRR}, for $i\in \ic{1,d}$, is made of exposed points of the unit
ball~$\TripleNormBall$.

\begin{proposition} 
  If the norm $\TripleNorm{\cdot}$ is orthant-strictly monotonic, 
  then the elements of the renormalized canonical basis of~$\RR^d$, that is
  the $e_i/ \TripleNorm{e_i}$ for $i\in \ic{1,d}$, 
  are exposed points of the unit ball~$\TripleNormBall$.
  \label{coro:osmexposed}
\end{proposition}

\begin{proof} 
  Assume that the norm $\TripleNorm{\cdot}$ is orthant-strictly
  monotonic and fix $i\in \ic{1,d}$. Then, using item~\ref{it:SBO} of
  Proposition~\ref{pr:orthant-strictly_monotonic}, we have that
  \( \TripleNorm{\overline{e}_i + \sum_{j \in \ic{1,d}\backslash\na{i}}
    \lambda_j \overline{e}_j} > \TripleNorm{\overline{e}_i}\), 
  for all $\ba{\lambda_j}_{j \in \ic{1,d}\backslash\na{i}}$ where not all
  $\lambda_j$'s are $0$ and where $\overline{e}_j = {e}_j/ \TripleNorm{e_j}$ for
  all $j\in \ic{1,d}$.  This means that the renormalized canonical basis is
  strongly orthonormal relative to $\overline{e}_i$ in the sense of
  Birkhoff. Using~\cite[Theorem 2.6]{MR3066832}, 
  we obtain that $\overline{e}_i$ is
  an exposed point of the unit ball~$\TripleNormBall$.
  This ends the proof.
\end{proof}

We recall that an \emph{extreme point}~$\primal$ of a convex set~$\Convex$ cannot be written
as \( \primal= \lambda \primal' + (1-\lambda) \primal'' \) with
\( \primal' \in \Convex \),   \( \primal'' \in \Convex \),
\( \primal' \neq \primal \), \( \primal'' \neq \primal \) and $\lambda\in ]0,1[$.
The normed space 
\( \bp{\RR^d,\TripleNorm{\cdot}} \) is said to be \emph{strictly convex}
if the unit ball~$\TripleNormBall$ (of the norm~$\TripleNorm{\cdot}$) is \emph{rotund}, that is,
if all points of the unit sphere~$\TripleNormSphere$ are extreme points
of the unit ball~$\TripleNormBall$.
The normed space \( \bp{\RR^d,\norm{\cdot}_{p}} \), equipped with 
the $\ell_p$-norm~$\norm{\cdot}_{p}$ (for $p\in [1,\infty]$), is strictly convex
if and only if $p\in ]1,\infty[$. 

\begin{proposition}
  If the norm $\TripleNorm{\cdot}$ is orthant-monotonic 
  and if the normed space 
  \( \bp{\RR^d,\TripleNorm{\cdot}} \) is strictly convex,
  then the norm~$\TripleNorm{\cdot}$ is orthant-strictly monotonic.
  \label{pr:orthant-monotonic+rotund}
\end{proposition}

\begin{proof}
  In~\cite[Theorem 2.2]{Sain-Paul-Jha:2015}, we find the following result:
  if the family \( \sequence{ \FlatRR_{K} }{K \subset \ic{1,d}} \)
  of subspaces of~$\RR^d$ is Birkhoff orthogonal 
  for a norm $\TripleNorm{\cdot}$,
  and if the unit ball for that norm is rotund,
  then the family \( \sequence{ \FlatRR_{K} }{K \subset \ic{1,d}} \) is
  strictly Birkhoff orthogonal. 

  Now for the proof.
  If the norm $\TripleNorm{\cdot}$ is orthant-monotonic, then 
  the family \( \sequence{ \FlatRR_{K} }{K \subset \ic{1,d}} \)
  of subspaces of~$\RR^d$ is Birkhoff orthogonal
  by Item~\ref{it:Birkhoff_orthogonal_IN_pr:orthant-monotonic}
  in Proposition~\ref{pr:orthant-monotonic}.
  As the unit ball for that norm is rotund, we deduce that 
  the family \( \sequence{ \FlatRR_{K} }{K \subset \ic{1,d}} \) is
  strictly Birkhoff orthogonal. 
  As Item~\ref{it:SBO} implies Item~\ref{it:OSM} 
  in Proposition~\ref{pr:orthant-strictly_monotonic},
  we conclude that the norm $\TripleNorm{\cdot}$ is orthant-strictly monotonic.
\end{proof}

\section{Generalized top-$k$ and $k$-support norms}
\label{Generalized_top-k_and_k-support_norms}

Let $\TripleNorm{\cdot}$ be a norm on~$\RR^d$, 
that we call the \emph{source norm}.
In~\S\ref{Definitions_of_generalized_top-k_and_k-support_norms},
we introduce generalized top-$k$ and $k$-support norms
constructed from the source norm,
and we provide various examples.
In~\S\ref{Properties_of_generalized_top-k_and_k-support_norms},
we establish properties valid for any source norm,
whereas, in~\S\ref{Properties_of_generalized_top-k_and_k-support_norms_under_orthant-monotonicity},
we establish properties valid when the source norm is orthant-monotonic, 
making thus the connection with the previous
Sect.~\ref{Orthant-monotonic_and_orthant-strictly_monotonic_norms}.

\subsection{Definition and examples}
\label{Definitions_of_generalized_top-k_and_k-support_norms}

We introduce generalized top-$k$ and $k$-support norms
that are constructed from the source norm~$\TripleNorm{\cdot}$. 

\begin{definition}
  \label{de:top_norm}
  For \( k \in \ic{1,d} \), we call \emph{generalized top-$k$ norm}
  (associated with the source norm~$\TripleNorm{\cdot}$)
  the norm defined by\footnote{%
    The notation \( \sup_{\cardinal{K} \leq k} \) is a shorthand for 
    \( \sup_{ { K \subset \ic{1,d}, \cardinal{K} \leq k}} \).}
  \begin{equation}
    \TopNorm{\TripleNorm{\primal}}{k}
    =
    \sup_{\cardinal{K} \leq k} \TripleNorm{\primal_K} 
    \eqsepv \forall \primal \in \RR^d 
    \eqfinp 
    \label{eq:top_norm}
  \end{equation}
  %
  We call \emph{generalized $k$-support norm}
  the dual norm of the generalized top-$k$ norm, denoted by\footnote{%
    We use the symbol~$\star$ in the superscript to indicate that the generalized
    $k$-support norm  \( \SupportNorm{\TripleNorm{\cdot}}{k} \) is a dual norm.
    To stress the point, we use the letter~$\primal$ for a primal vector,
    like in~\( \TopNorm{\TripleNorm{\primal}}{k} \),
    and the letter~$\dual$ for a dual vector,
    like in~\( \SupportNorm{\TripleNorm{\dual}}{k} \). }
  \( \SupportNorm{\TripleNorm{\cdot}}{k} \):
  \begin{equation}
    \SupportNorm{\TripleNorm{\cdot}}{k} 
    = \Bp{ \TopNorm{\TripleNorm{\cdot}}{k} }_{\star}
    \eqfinp
    \label{eq:support_norm}
  \end{equation}
\end{definition}
It is easily verified that \( \TopNorm{\TripleNorm{\cdot}}{k} \) indeed is a
norm, for all~\( k \in \ic{1,d} \).


We provide examples 
of generalized top-$k$ and $k$-support norms
in the case of permutation invariant monotonic source norms
and of $\ell_p$ source norms.
Table~\ref{tab:Examples} provides a summary.

\subsubsubsection{The case of permutation invariant monotonic source norms}

Letting \( \primal\in\RR^d \) and
$\nu$ be a permutation of \( \ic{1,d} \) such that
\( \module{ \primal_{\nu(1)} } \geq \module{ \primal_{\nu(2)} } 
\geq \cdots \geq \module{ \primal_{\nu(d)} } \),
we note \( \primal^{\downarrow} = \bp{ \module{ \primal_{\nu(1)} }, 
  \module{ \primal_{\nu(2)} }, \ldots, \module{ \primal_{\nu(d)} } } \).
The proof of the following Lemma is easy.
\begin{lemma}
  \label{lem:permutation_invariant_monotonic} 
  Let $\TripleNorm{\cdot}$ be a norm on~$\RR^d$.
  Then, if the norm \( \TripleNorm{\cdot} \) is
  permutation invariant and monotonic,
  we have that $\TopNorm{\TripleNorm{\primal}}{k} = 
  \TripleNorm{ \primal^{\downarrow}_{ \{ 1,\ldots,k \} }}  $,
  where $\primal^{\downarrow}_{ \{ 1,\ldots,k \} } \in \RR^d$ is given by
  $\np{\primal^{\downarrow}}_{\{ 1,\ldots,k \}}$,
  for all $\primal \in \RR^d$.
\end{lemma}

\subsubsubsection{The case of $\ell_p$ source norms}

We start with generalized top-$k$ norms as in~\eqref{eq:top_norm}
(see the second column of Table~\ref{tab:Examples}). 
When the norm~$\TripleNorm{\cdot}$ is the Euclidean norm~\( \norm{\cdot}_2 \) of~\( \RR^d \),
the generalized top-$k$ norm is known under different names:
the top-$(k,2)$ norm in~\cite[p.~8]{Tono-Takeda-Gotoh:2017},
or the $2$-$k$-symmetric gauge norm~\cite{Mirsky:1960}
or the Ky Fan vector norm~\cite{Obozinski-Bach:hal-01412385}.
Indeed, in all these cases, the norm of a vector~$\primal$ is obtained
with a subvector of size~$k$ having the~$k$ largest components in module,
because the assumptions of Lemma~\ref{lem:permutation_invariant_monotonic} are satisfied.
More generally, when the norm~$\TripleNorm{\cdot}$ is the $\ell_p$-norm
$\norm{\cdot}_{p}$, for $p\in [1,\infty]$, 
the assumptions of Lemma~\ref{lem:permutation_invariant_monotonic} are also satisfied, as 
$\ell_p$-norms are permutation invariant and monotonic. 
Therefore, we obtain 
that the corresponding generalized top-$k$ norm
\( \TopNorm{ \bp{ \norm{\cdot}_{p} } }{k} \) has the expression 
\( \TopNorm{ \bp{ \norm{\cdot}_{p} } }{k}\np{\primal}
=  \sup_{\cardinal{K} \leq k} {\norm{\primal_K}}_{p} 
= {\norm{ \primal^{\downarrow}_{ \{ 1,\ldots,k \} }}}_{p} \),
for all \( \primal \in \RR^d \).
Thus, we have obtained that the generalized top-$k$ norm associated with the
$\ell_p$-norm is the norm
${\norm{ ({\cdot})^{\downarrow}_{ \{ 1,\ldots,k \} }}}_{p}$:
we call it\footnote{We invert the indices in the naming convention
  of~\cite[p.~5, p.~8]{Tono-Takeda-Gotoh:2017}, where top-$(k,1)$ and top-$(k,2)$
  were used.}
\emph{\lptopnorm{p}{k}} and we denote it by~\( \LpTopNorm{\cdot}{p}{k} \).
Notice that \( \LpTopNorm{\cdot}{\infty}{k} 
= \norm{\cdot}_{\infty}\) for all~$k\in\ic{1,d}$.
\medskip

Now, we turn to generalized $k$-support norms as in~\eqref{eq:support_norm}
(see the third column of Table~\ref{tab:Examples}).
When the norm~$\TripleNorm{\cdot}$ is the Euclidean norm~\( \norm{\cdot}_2 \) of~\( \RR^d \),
the generalized $k$-support norm is the so-called
$k$-support norm~\cite{Argyriou-Foygel-Srebro:2012}. 
More generally, in~\cite[Definition 21]{McDonald-Pontil-Stamos:2016},
the authors define the $k$-support $p$-norm or \emph{\lpsupportnorm{p}{k}}
for $p\in [1,\infty]$.
They show, in~\cite[Corollary 22]{McDonald-Pontil-Stamos:2016},
that the dual norm \( \Bp{ \TopNorm{ \bp{ \norm{\cdot}_{p} } }{k} }_\star \) 
of the above top-$(k,p)$ norm
is the $(q,k)$-support norm, where \( 1/p + 1/q = 1 \).
Thus, what we call the generalized $k$-support norm 
\( \SupportNorm{ \bp{ \norm{\cdot}_{p} } }{k} =
\Bp{ \TopNorm{ \bp{ \norm{\cdot}_{p} } }{k} }_\star \) 
associated with the $\ell_p$-norm
is the \lpsupportnorm{q}{k}, that we denote~\( \LpSupportNorm{\dual}{q}{k} \).
The formula \( \LpSupportNorm{\primal}{\infty}{k} =
\max \na{ \Norm{\primal}_{1} / k , \Norm{\primal}_{\infty} } \)
can be found in~\cite[Exercise IV.1.18, p. 90]{Bhatia:1997}.

\begin{table}
  \centering
  \begin{tabular}{||c||c|c||}
    \hline\hline 
    source norm~\( \TripleNorm{\cdot} \) 
    & \( \TopNorm{\TripleNorm{\cdot}}{k} \)
    & \( \SupportNorm{\TripleNorm{\cdot}}{k} \)
    \\
    \hline\hline 
    \( \norm{\cdot}_{p} \)
    & \lptopnorm{p}{k} 
    & \lpsupportnorm{q}{k} 
    \\
    & \( \LpTopNorm{\primal}{p}{k} \) 
    & \( \LpSupportNorm{\dual}{q}{k} \) 
    \\
    & \( =\bp{ \sum_{\LocalIndex=1}^{k} \module{ \primal_{\nu(\LocalIndex)} }^p }^{1/p} \)
    & \( 1/p + 1/q =1 \)
    \\
    \hline
    \( \norm{\cdot}_{1} \) 
    & \lptopnorm{1}{k} 
    & \lpsupportnorm{\infty}{k} 
    \\
    & \( \LpTopNorm{\primal}{1}{k} = \sum_{l=1}^{k} \module{ \primal_{\nu(l)} } \) 
    & \( \LpSupportNorm{\dual}{\infty}{k} =
      \max \na{ \Norm{\dual}_{1} / k , \Norm{\dual}_{\infty} } \) 
    \\
    \hline
    \( \norm{\cdot}_{2} \) 
    & \lptopnorm{2}{k} 
    & \lpsupportnorm{2}{k} 
    \\
    & \( \LpTopNorm{\primal}{2}{k} = \sqrt{ \sum_{l=1}^{k} \module{ \primal_{\nu(l)} }^2 } \)
    & \( \LpSupportNorm{\dual}{2}{k} \) no analytic expression
    \\
    & & (computation in~\cite[Prop.~2.1]{Argyriou-Foygel-Srebro:2012})
    \\
    \hline 
    \( \norm{\cdot}_{\infty} \) 
    & \lptopnorm{\infty}{k} 
    & \lpsupportnorm{1}{k} 
    \\
    & $\ell_{\infty}$-norm 
    & $\ell_{1}$-norm 
    \\
    & \( \LpTopNorm{\primal}{\infty}{k} = \module{ \primal_{\nu(1)} } = \Norm{\primal}_{\infty} \) 
    & \( \LpSupportNorm{\dual}{1}{k} = \Norm{\dual}_{1} \) 
    \\
    \hline\hline
  \end{tabular}
  \caption{Examples of generalized top-$k$ and $k$-support norms
    generated by the $\ell_p$ source norms 
    \( \TripleNorm{\cdot} = \norm{\cdot}_{p} \) for $p\in [1,\infty]$;
    $\nu$ is a permutation of \( \ic{1,d} \) such that
    \( \module{ \primal_{\nu(1)} } \geq \module{ \primal_{\nu(2)} } 
    \geq \cdots \geq \module{ \primal_{\nu(d)} } \)
    \label{tab:Examples}}
\end{table}

\subsection{General properties} 
\label{Properties_of_generalized_top-k_and_k-support_norms}

We establish properties of generalized top-$k$ and $k$-support norms,
valid for any source norm, 
that will be useful to prove our results 
in~Sect.~\ref{The_lzeropseudonorm_orthant-monotonicity_and_generalized_top-k_and_k-support_norms}.

\subsubsection*{Properties of generalized top-$k$ norms}

We denote the unit ball 
of the generalized top-$k$ norm
\( \TopNorm{\TripleNorm{\cdot}}{k} \) 
in Definition~\ref{de:top_norm} by
\begin{align}
  \TopNorm{\TripleNormBall}{k} 
  &= 
    \bset{\primal \in \RR^d}{\TopNorm{\TripleNorm{\primal}}{k} \leq 1} 
    \eqsepv \forall k\in\ic{1,d}
    \eqfinp
    \label{eq:generalized_top-k_norm_unit_ball}
\end{align}

\begin{proposition}\quad
  \begin{itemize}
  \item 
    For \( k \in \ic{1,d} \), the generalized top-$k$ norm
    \( \TopNorm{\TripleNorm{\cdot}}{k} \) 
    (in Definition~\ref{de:top_norm}) has the expression
    \begin{equation}
      \TopNorm{\TripleNorm{\primal}}{k}
      =
      \sup_{\cardinal{K} \leq k} 
      \sigma_{\pi_K\np{\TripleNormDualSphere}}\np{\primal} 
      \eqsepv \forall \primal \in \RR^d 
      \eqfinv
      \label{eq:generalized_top-k_norm_equality}
    \end{equation}
    where $\TripleNormDualSphere$ is the unit sphere 
    of the dual norm~$\TripleNormDual{\cdot}$ as
    in~\eqref{eq:triplenorm_Dual_unit_sphere}.
  \item 
    We have the inequality
    \begin{equation}
      \TripleNorm{\primal} \leq \TopNorm{\TripleNorm{\primal}}{d} 
      \eqsepv \forall \primal \in \RR^d 
      \eqfinp  
      \label{eq:generalized_top-d_norm_inequality}
    \end{equation}
  \item 
    The sequence \( \bseqa{\TopNorm{\TripleNorm{\cdot}}{\LocalIndex}}{\LocalIndex\in\ic{1,d}} \)
    of generalized top-$k$ norms in~\eqref{eq:top_norm} is nondecreasing, 
    in the sense that the following inequalities hold true
    \begin{equation}
      \TopNorm{\TripleNorm{\primal}}{1} \leq \cdots \leq 
      \TopNorm{\TripleNorm{\primal}}{\LocalIndex} \leq 
      \TopNorm{\TripleNorm{\primal}}{\LocalIndex+1} \leq \cdots \leq 
      \TopNorm{\TripleNorm{\primal}}{d} 
      \eqsepv \forall \primal \in \RR^d 
      \eqfinp
      \label{eq:generalized_top-k_norm_inequalities}
    \end{equation}
  \item 
    The sequence 
    \( \bseqa{ \TopNorm{\TripleNormBall}{\LocalIndex}}{\LocalIndex\in\ic{1,d}} \)
    of units balls of the generalized top-$k$ norms 
    in~\eqref{eq:generalized_top-k_norm_unit_ball} is nonincreasing, 
    in the sense that the following inclusions hold true:
    \begin{equation}
      \TopNorm{\TripleNormBall}{d} 
      \subset \cdots \subset \TopNorm{\TripleNormBall}{\LocalIndex+1}
      \subset \TopNorm{\TripleNormBall}{\LocalIndex}\subset \cdots 
      \subset \TopNorm{\TripleNormBall}{1}
      \eqfinp 
      \label{eq:generalized_top-k_norm_unit_balls_inclusions}
    \end{equation}
  \end{itemize}
\end{proposition}

\begin{proof}
  
  \noindent $\bullet$ 
  For any \( \primal \in \RR^d \), we have
  \begin{align*}
    \TopNorm{\TripleNorm{\primal}}{k} 
    &=
      \sup_{\cardinal{K} \leq k} \TripleNorm{\primal_K}
      \intertext{by definition~\eqref{eq:top_norm} of the generalized top-$k$ norm
      \( \TopNorm{\TripleNorm{\cdot}}{k} \) }
    &=
      \sup_{\cardinal{K} \leq k} 
      \sigma_{ \TripleNormDualSphere }\np{\primal_K} 
      \tag{by~\eqref{eq:norm_dual_norm}}
    \\
    &= 
      \sup_{\cardinal{K} \leq k} 
      \sup_{\dual \in \TripleNormDualSphere} \proscal{\primal_K}{\dual} 
      \tag{by definition~\eqref{eq:support_function} of the support function~$\sigma_{\TripleNormDualSphere}$ }
    \\
    &= 
      \sup_{\cardinal{K} \leq k} 
      \sup_{\dual \in \TripleNormDualSphere} \proscal{\primal}{\pi_K\np{\dual}}
      \tag{by the self-duality property~\eqref{eq:orthogonal_projection_self-dual} of 
      the projection mapping~\( \pi_K \) }
    \\ 
    &= 
      \sup_{\cardinal{K} \leq k} 
      \sup_{\dual' \in \pi_K\np{\TripleNormDualSphere}} \proscal{\primal}{\dual'}
    \\
    &= 
      \sup_{\cardinal{K} \leq k} 
      \sigma_{\pi_K\np{\TripleNormDualSphere}}\np{\primal} 
      \tag{by definition~\eqref{eq:support_function} of the support function~$\sigma_{\pi_K\np{\TripleNormDualSphere}}$
      }
  \end{align*}
  and we get~\eqref{eq:generalized_top-k_norm_equality}.
  \medskip

  \noindent $\bullet$ 
  From the very definition~\eqref{eq:top_norm}
  of the generalized top-$d$ norm \( \TopNorm{\TripleNorm{\cdot}}{d} \),
  we get
  \[
    \TopNorm{\TripleNorm{\primal}}{d} =
    \sup_{\cardinal{K} \leq d} \TripleNorm{\primal_K} 
    \geq \TripleNorm{\primal_{\ic{1,d}}} =  \TripleNorm{\primal}
    \eqsepv \forall \primal \in \RR^d \eqfinv
  \]
  hence~\eqref{eq:generalized_top-d_norm_inequality}. 
  \medskip

  \noindent $\bullet$ 
  The inequalities~\eqref{eq:generalized_top-k_norm_inequalities} 
  between norms easily derive 
  from the very definition~\eqref{eq:top_norm}
  of the generalized top-$k$ norms
  \( \TopNorm{\TripleNorm{\cdot}}{k} \).
  \medskip

  \noindent $\bullet$ 
  The inclusions~\eqref{eq:generalized_top-k_norm_unit_balls_inclusions}
  between unit balls
  directly follow from the
  inequalities~\eqref{eq:generalized_top-k_norm_inequalities}
  between norms.
  \medskip

  This ends the proof. 
\end{proof}

\subsubsection*{Properties of generalized $k$-support norms}

We denote the unit ball 
of the generalized $k$-support norm~\( \SupportNorm{\TripleNorm{\cdot}}{k} \)
in Definition~\ref{de:top_norm} by
\begin{align}
  \SupportNorm{\TripleNormBall}{k} 
  &  = 
    \defset{\dual \in \RR^d}{\SupportNorm{\TripleNorm{\dual}}{k} \leq 1} 
    \eqsepv \forall k\in\ic{1,d}
    \eqfinp
    \label{eq:generalized_k-support_norm_unit_ball}
\end{align}

\begin{proposition}
  \quad
  \begin{itemize}
  \item 
    For \( k \in \ic{1,d} \), the generalized $k$-support norm~\( \SupportNorm{\TripleNorm{\cdot}}{k} \)
    in Definition~\ref{de:top_norm} has unit ball
    \begin{equation}
      \SupportNorm{\TripleNormBall}{k} =
      \closedconvexhull\bp{ \bigcup_{ \cardinal{K} \leq k} 
        \pi_{K} \np{ \TripleNormDualSphere } }
      \eqfinv
      \label{eq:generalized_k-support_norm_unit_ball_property}
    \end{equation}
    where \( \closedconvexhull\np{S} \) denotes the closed convex hull
    of a subset \( S \subset \RR^d \).
  \item 
    We have the inequality
    \begin{equation}
      \SupportNorm{\TripleNorm{\dual}}{d}
      \leq \TripleNormDual{\dual} 
      \eqsepv \forall \dual \in \RR^d 
      \eqfinp  
      \label{eq:generalized_d-support_norm_inequality}
    \end{equation}
  \item 
    The sequence 
    \( \bseqa{\SupportNorm{\TripleNorm{\cdot}}{\LocalIndex}}{\LocalIndex\in\ic{1,d}} \)
    of generalized $k$-support norms in~\eqref{eq:support_norm} is nonincreasing, 
    in the sense that the following inequalities hold true
    \begin{equation}
      \SupportNorm{\TripleNorm{\dual}}{d}
      \leq \cdots \leq 
      \SupportNorm{\TripleNorm{\dual}}{\LocalIndex+1}\leq 
      \SupportNorm{\TripleNorm{\dual}}{\LocalIndex}
      \leq \cdots \leq 
      \SupportNorm{\TripleNorm{\dual}}{1} 
      \eqsepv \forall \dual \in \RR^d 
      \eqfinp
      \label{eq:generalized_k-support_norm_inequalities}
    \end{equation}
  \item 
    The sequence 
    \( \bseqa{ \SupportNorm{\TripleNormBall}{\LocalIndex}}{\LocalIndex\in\ic{1,d}} \)
    of units balls of the generalized $k$-support norms 
    in~\eqref{eq:generalized_k-support_norm_unit_ball} is nondecreasing, 
    in the sense that the following inclusions hold true:
    \begin{equation}
      \SupportNorm{\TripleNormBall}{1} 
      \subset \cdots \subset
      \SupportNorm{\TripleNormBall}{\LocalIndex} \subset 
      \SupportNorm{\TripleNormBall}{\LocalIndex+1} 
      \subset \cdots \subset \SupportNorm{\TripleNormBall}{d} 
      \eqfinp 
      \label{eq:generalized_k-support_norm_unit_balls_inclusions}
    \end{equation}
  \end{itemize}
\end{proposition}

\begin{proof}

  \noindent $\bullet$
  For any \( \primal \in \RR^d \), we have
  \begin{align*}
    \TopNorm{\TripleNorm{\primal}}{k} 
    &=
      \sup_{\cardinal{K} \leq k} 
      \sigma_{\pi_K\np{\TripleNormDualSphere}}\np{\primal} 
      \tag{by~\eqref{eq:generalized_top-k_norm_equality}}
    \\
    &=
      \sigma_{ \bigcup_{ \cardinal{K} \leq k} 
      \pi_{K}\np{ \TripleNormDualSphere } } \np{\primal}
      \tag{as the support function turns a union of sets into a supremum}  
    \\
    &=
      \sigma_{ \closedconvexhull\bp{ \bigcup_{ \cardinal{K} \leq k} 
      \pi_{K}\np{ \TripleNormDualSphere } } }\np{\primal}
      \tag{by~\cite[Prop.~7.13]{Bauschke-Combettes:2017} }
  \end{align*}
  and we obtain~\eqref{eq:generalized_k-support_norm_unit_ball_property}
  thanks to~\eqref{eq:norm_dual_norm}.
  \medskip

  \noindent $\bullet$ 
  From the inequality~\eqref{eq:generalized_top-d_norm_inequality} between norms,
  we deduce the inequality~\eqref{eq:generalized_d-support_norm_inequality}
  between dual norms, by the definition~\eqref{eq:dual_norm} of a dual norm.
  \medskip

  \noindent $\bullet$ 
  The inequalities in~\eqref{eq:generalized_k-support_norm_inequalities} easily derive 
  from the inclusions~\eqref{eq:generalized_k-support_norm_unit_balls_inclusions}.
  \medskip

  \noindent $\bullet$ 
  The  inclusions~\eqref{eq:generalized_k-support_norm_unit_balls_inclusions}
  directly follow from the 
  inclusions~\eqref{eq:generalized_top-k_norm_unit_balls_inclusions}
  and from~\eqref{eq:norm_dual_norm_unit_ball}
  as \( \SupportNorm{\TripleNormBall}{k} = 
  \bp{\TopNorm{\TripleNormBall}{k}}^{\odot} \),
  the polar set of~\( \TopNorm{\TripleNormBall}{k} \).

  This ends the proof. 
\end{proof}

\subsection{Properties under orthant-monotonicity}
\label{Properties_of_generalized_top-k_and_k-support_norms_under_orthant-monotonicity}

We establish properties of generalized top-$k$ and $k$-support norms,
valid when the source norm is orthant-monotonic, 
that will be useful to prove our results in~Sect.~\ref{The_lzeropseudonorm_orthant-monotonicity_and_generalized_top-k_and_k-support_norms}.

\begin{proposition} \label{pr:source_norm_orthant-monotonic_generalized_top-k_norm} 
  \quad
  \begin{enumerate}
  \item 
    Let \( k \in \ic{1,d} \).
    If the source norm~$\TripleNorm{\cdot}$ is orthant-monotonic,
    then 
    \begin{itemize}
    \item 
      the generalized top-$k$ norm has the expression
      \begin{equation}
        \TopNorm{\TripleNorm{\primal}}{k}
        =
        \sup_{\cardinal{K} \leq k} 
        \sigma_{ \FlatRR_{K} \cap \TripleNormDualSphere } \np{\primal}
        \eqsepv \forall \primal \in \RR^d 
        \eqfinv
        \label{eq:TopNorm=CoordinateNorm}
      \end{equation}
      where $\TripleNormDualSphere$ is the unit sphere 
      of the dual norm~$\TripleNormDual{\cdot}$ as
      in~\eqref{eq:triplenorm_Dual_unit_sphere},
    \item 
      the unit ball of the $k$-support norm is given by 
      \begin{equation}
        \SupportNorm{\TripleNormBall}{k} 
        = \closedconvexhull\bp{ \bigcup_{ \cardinal{K} \leq k} 
          \np{ \FlatRR_{K} \cap \TripleNormDualSphere } }
        \eqfinp
        \label{eq:dual_support_norm_unit_ball}
      \end{equation}
    \end{itemize}
  \item 
    The source norm~$\TripleNorm{\cdot}$ is orthant-monotonic
    if and only if 
    \( \TripleNorm{\cdot} =\TopNorm{\TripleNorm{\cdot}}{d} \)
    if and only if 
    \( \TripleNormDual{\cdot} =\SupportNorm{\TripleNorm{\cdot}}{d} \).
    \label{it:generalized_top-d_norm_equality}
  \item 
    If the source norm~$\TripleNorm{\cdot}$ is orthant-monotonic, then
    the generalized top-$k$ norms and the generalized $k$-support norms
    are orthant-monotonic.
    \label{it:generalized_top-ksupport_norm_orthant-monotonic}
  \end{enumerate}
\end{proposition}

\begin{proof}
  \quad
  \begin{enumerate}
  \item 
    We suppose that the source norm~$\TripleNorm{\cdot}$ is orthant-monotonic.
    Let \( k\in\ic{1,d} \).

    \noindent $\bullet$ 
    We prove~\eqref{eq:TopNorm=CoordinateNorm}. 
    For any \( \primal \in \RR^d \),
    we have 
    \begin{align*}
      \TopNorm{\TripleNorm{\primal}}{k}
      &=
        \sup_{\cardinal{K} \leq k} \TripleNorm{\primal_K} 
        \tag{by definition~\eqref{eq:top_norm} of the generalized top-$k$ norm}
      \\
      &=
        \sup_{\cardinal{K} \leq k} \TripleNorm{\primal_K}_{\star\star} 
        \tag{as any norm is equal to its bidual norm by~\eqref{eq:bidual_norm=norm} } 
      \\
      &=
        \sup_{\cardinal{K} \leq k} 
        \np{ \TripleNorm{\cdot}_{\star} }_{\star, K}\np{\primal_K}
        \tag{by Definition~\ref{de:K_norm} of the the $\StarSet{K}$-norm }
      \\
      &=
        \sup_{\cardinal{K} \leq k} 
        \np{ \TripleNorm{\cdot}_{\star} }_{K, \star}\np{\primal_K}
        \intertext{by Item~\ref{it:K_star_=_star_K_IN_pr:orthant-monotonic}
        in Proposition~\ref{pr:orthant-monotonic}
        because, as the norm~$\TripleNorm{\cdot}$ is orthant-monotonic, 
        so is also the dual norm~$\TripleNormDual{\cdot}$ 
        (equivalence between Item~\ref{it:orthant-monotonic_IN_pr:orthant-monotonic}
        and Item~\ref{it:dual_orthant-monotonic_IN_pr:orthant-monotonic} 
        in Proposition~\ref{pr:orthant-monotonic})}
      &=
        \sup_{\cardinal{K} \leq k} 
        \sigma_{ \FlatRR_{K} \cap \TripleNormDualSphere } \np{\primal_K}
        \tag{by~\eqref{eq:K_star=sigma} applied to $\TripleNorm{\cdot}_{\star}$ with $\primal_K\in \FlatRR_K$}
      \\
      &=
        \sup_{\cardinal{K} \leq k} 
        \sigma_{ \FlatRR_{K} \cap \TripleNormDualSphere } \np{\primal}
    \end{align*}
    by the self-duality property~\eqref{eq:orthogonal_projection_self-dual} of 
    the projection mapping~\( \pi_K \),
    and by definition~\eqref{eq:FlatRR} of the subspace~\( \FlatRR_{K} \).
    \medskip

    \noindent $\bullet$ 
    We prove~\eqref{eq:dual_support_norm_unit_ball}. 
    Indeed, by~\eqref{eq:TopNorm=CoordinateNorm},
    we have that 
    \( \TopNorm{\TripleNorm{\cdot}}{k}
    =  \sup_{\cardinal{K} \leq k} 
    \sigma_{ \FlatRR_{K} \cap \TripleNormDualSphere } \).
    As \(  \sup_{\cardinal{K} \leq k} 
    \sigma_{ \FlatRR_{K} \cap \TripleNormDualSphere } \) 
    \( = \sigma_{ \bigcup_{ \cardinal{K} \leq k} \np{ \FlatRR_{K} \cap \TripleNormDualSphere } } \), 
    we have just established that 
    \( \TopNorm{\TripleNorm{\cdot}}{k} =
    \sigma_{ \cup_{\cardinal{K} \leq k} \np{ \FlatRR_{K} \cap \TripleNormDualSphere } } \).
    On the other hand, by~\eqref{eq:norm_dual_norm} 
    we have that \( \TopNorm{\TripleNorm{\cdot}}{k}=
    \sigma_{ \SupportNorm{\TripleNormBall}{k} } \)
    since, by Definition~\ref{de:top_norm}, the $k$-support norm
    is the dual norm of the top-$k$ norm.
    Then, by~\cite[Prop.~7.13]{Bauschke-Combettes:2017}, 
    we deduce that
    \( 
    \closedconvexhull\bp{ \SupportNorm{\TripleNormBall}{k} } 
    =
    \closedconvexhull\bp{ \bigcup_{ \cardinal{K} \leq k} \np{ \FlatRR_{K} \cap \TripleNormDualSphere } }
    \).
    As the unit ball~\( \SupportNorm{\TripleNormBall}{k} \)
    in~\eqref{eq:generalized_k-support_norm_unit_ball}
    is closed and convex, we immediately
    obtain~\eqref{eq:dual_support_norm_unit_ball}. 

  \item
    First, let us observe that, from the very definition~\eqref{eq:top_norm}
    of the generalized top-$d$ norm \( \TopNorm{\TripleNorm{\cdot}}{d} \),
    and by~\eqref{eq:generalized_top-d_norm_inequality}, we have,
    for all $\primal \in \RR^d$:
    \begin{equation}
      \TopNorm{\TripleNorm{\primal}}{d} = \TripleNorm{\primal}
      \iff 
      \sup_{\cardinal{K} \leq d} \TripleNorm{\primal_K} = \TripleNorm{\primal}
      \iff 
      \TripleNorm{\primal_K} \leq \TripleNorm{\primal}
      \eqsepv \forall K \subset \ic{1,d} 
      \eqfinp
      \label{eq:generalized_top-d_norm_equality_in_the_proof}     
    \end{equation}
    Now, we turn to prove Item~\ref{it:generalized_top-d_norm_equality}
    as two reverse implications.

    Suppose that the source norm~$\TripleNorm{\cdot}$ is orthant-monotonic,
    and let us prove that
    \( \TopNorm{\TripleNorm{\primal}}{d} = \TripleNorm{\primal} \).
    By Item~\ref{it:ICS} in Proposition~\ref{pr:orthant-monotonic},
    we get that \( \TripleNorm{\primal_K} \leq \TripleNorm{\primal} \),
    for all \( K \subset \ic{1,d} \),
    hence \( \TopNorm{\TripleNorm{\primal}}{d} = \TripleNorm{\primal} \),
    for all $\primal \in \RR^d$
    by the just proven equivalence~\eqref{eq:generalized_top-d_norm_equality_in_the_proof}. 

    Suppose that \( \TopNorm{\TripleNorm{\primal}}{d}
    = \TripleNorm{\primal} \) 
    and let us prove that
    the source norm~$\TripleNorm{\cdot}$ is orthant-monotonic.
    By~\eqref{eq:generalized_top-d_norm_equality_in_the_proof}, we have that 
    \( \TripleNorm{\primal_J} \leq \TripleNorm{\primal} \),
    for all $\primal \in \RR^d$ and all \( J \subset \ic{1,d} \).
    This gives, in particular, 
    \( \TripleNorm{\primal_{K \cap J}}=\TripleNorm{\np{\primal_K}_J}
    \leq \TripleNorm{\primal_K} \);
    if \( J \subset K \), we deduce that 
    \( \TripleNorm{\primal_J}  \leq \TripleNorm{\primal_K} \).
    Thus, Item~\ref{it:ICS} in Proposition~\ref{pr:orthant-monotonic} holds true,
    and we obtain that the source norm~$\TripleNorm{\cdot}$ is orthant-monotonic.

    We end the proof by taking the dual norms, as in~\eqref{eq:dual_norm}, of both sides of the equality
    \( \TripleNorm{\cdot} =\TopNorm{\TripleNorm{\cdot}}{d} \),
    yielding 
    \( \TripleNormDual{\cdot} =\SupportNorm{\TripleNorm{\cdot}}{d} \)
    by~\eqref{eq:support_norm}. 
  \item 
    The generalized top-$k$ norm in~\eqref{eq:top_norm} 
    is the supremum of the subfamily,
    when $\cardinal{K} \leq k$, of the 
    seminorms~\( \TripleNorm{\pi_K\np{\cdot}}_{K} \).
    As already mentioned, the definition of orthant-monotonic norms
    can be extended to seminorms. With this extension, it is easily seen that the 
    seminorms~\( \TripleNorm{\pi_K\np{\cdot}}_{K} \) are orthant-monotonic
    as soon as the source norm~$\TripleNorm{\cdot}$ is orthant-monotonic.
    Therefore, if the source norm~$\TripleNorm{\cdot}$ is orthant-monotonic,
    so is the supremum in~\eqref{eq:top_norm}, 
    thanks to the property
    claimed right after the Definition~\ref{de:orthant-monotonic}:
    the supremum of a family of orthant-monotonic seminorms is an
    orthant-monotonic seminorm.
    Thus, we have established that the generalized top-$k$ norm in~\eqref{eq:top_norm} 
    is orthant-monotonic.
    We deduce that its dual norm, 
    the generalized $k$-support norm~\( \SupportNorm{\TripleNorm{\cdot}}{k} \)
    in~\eqref{eq:support_norm},
    is orthant-monotonic. Indeed, the dual norm
    of an orthant-monotonic norm~$\TripleNorm{\cdot}$ is orthant-monotonic, 
    as proved in~\cite[Theorem~2.23]{Gries:1967} 
    (equivalence between Item~\ref{it:orthant-monotonic_IN_pr:orthant-monotonic}
    and Item~\ref{it:dual_orthant-monotonic_IN_pr:orthant-monotonic} 
    in Proposition~\ref{pr:orthant-monotonic}).
  \end{enumerate}
  \medskip

  This ends the proof.
\end{proof}

\section{The \lzeropseudonorm, orthant-monotonicity and generalized top-$k$ and $k$-support norms}
\label{The_lzeropseudonorm_orthant-monotonicity_and_generalized_top-k_and_k-support_norms}

In~\S\ref{The_lzeropseudonorm_and_its_level_sets}, 
we introduce basic notation regarding the \lzeropseudonorm.
In~\S\ref{Graded_sequences_of_norms},
we introduce the notions of (strictly) increasingly or decreasingly graded sequences of norms,
and we display conditions for generalized top-$k$ norms or 
generalized $k$-support norms to be graded sequences.

\subsection{Level sets of the \lzeropseudonorm}
\label{The_lzeropseudonorm_and_its_level_sets}

The so-called \emph{\lzeropseudonorm} is the function
\( \lzero : \RR^d \to \ic{0,d} \)
defined, for any \( \primal \in \RR^d \), by
\begin{equation}
  \lzero\np{\primal} = \cardinal{ \Support{\primal} }
  = \textrm{number of nonzero components of } \primal
  \eqfinp
  \label{eq:pseudo_norm_l0}  
\end{equation}
The \lzeropseudonorm\ shares three out of the four axioms of a norm:
nonnegativity, positivity except for \( \primal =0 \), subadditivity.
The axiom of 1-homogeneity does not hold true;
by contrast, the \lzeropseudonorm\ is 0-homogeneous:
\begin{equation}
  \lzero\np{\rho\primal} = \lzero\np{\primal} 
  \eqsepv \forall \rho \in \RR\backslash\{0\}
  \eqsepv \forall \primal \in \RR^d
  \eqfinp
  \label{eq:lzeropseudonorm_is_0-homogeneous}
\end{equation}
We introduce the \emph{level sets}
\begin{equation}
  \LevelSet{\lzero}{k} 
  = 
  \bset{ \primal \in \RR^d }{ \lzero\np{\primal} \leq k }
  \eqsepv \forall k\in\ic{0,d} 
  \eqfinp
  \label{eq:pseudonormlzero_level_set}
\end{equation}
The level sets of the \lzeropseudonorm\
in~\eqref{eq:pseudonormlzero_level_set}
are easily related to the subspaces~\( \FlatRR_{K} \) of~\( \RR^d \),
as defined in~\eqref{eq:FlatRR}, by
\begin{equation}
  \LevelSet{\lzero}{k} 
  = 
  \defset{ \primal \in \RR^d }{ \lzero\np{\primal} \leq k }
  = \bigcup_{\cardinal{K} \leq k} \FlatRR_{K} 
  \eqsepv \forall k\in\ic{0,d} 
  \eqfinv
  \label{eq:level_set_pseudonormlzero}
\end{equation}
where the notation \( \bigcup_{\cardinal{K} \leq k} \)
is a shorthand for 
\( \bigcup_{ { K \subset \ic{1,d}, \cardinal{K} \leq k}} \).

If the source norm~$\TripleNorm{\cdot}$ is orthant-monotonic,
the expression~\eqref{eq:dual_support_norm_unit_ball} of
the unit ball of the $k$-support norm can be written
with the level sets of the \lzeropseudonorm\ as 
\begin{equation}
  \SupportNorm{\TripleNormBall}{k} 
  = \closedconvexhull\bp{ \bigcup_{ \cardinal{K} \leq k} 
    \np{ \FlatRR_{K} \cap \TripleNormDualSphere } }
  = \closedconvexhull\bp{ \LevelSet{\lzero}{k} \cap \TripleNormDualSphere }
  \eqfinp
  \label{eq:dual_support_norm_unit_ball_bis}
\end{equation}
This formula is reminiscent of (and generalizes)
\cite[Equation~(2)]{Argyriou-Foygel-Srebro:2012},
which was established for the Euclidean source norm.
With an additional assumption, we obtain a refinement.
The proof of the following 
Proposition~\ref{pr:level_set_pseudonormlzero_intersection_sphere_rotund} 
relies on 
Lemma~\ref{lemma:convex_env} and its Corollary~\ref{cor:convex_env}.

\begin{proposition}  
  \label{pr:level_set_pseudonormlzero_intersection_sphere_rotund}
  If the source norm~$\TripleNorm{\cdot}$ is orthant-monotonic
  and if the normed space 
  \( \bp{\RR^d,\TripleNormDual{\cdot}} \) is strictly convex,
  then we have
  \begin{equation}
    \LevelSet{\lzero}{k} \cap \TripleNormDualSphere 
    =
    \SupportNorm{\TripleNormBall}{k} \cap \TripleNormDualSphere 
    \eqsepv \forall k \in \ic{0,d} 
    \eqfinv
    \label{eq:level_set_l0_inter_sphere_b}
  \end{equation}
  where \( \LevelSet{\lzero}{k} \) is the level set
  in~\eqref{eq:pseudonormlzero_level_set}
  of the \lzeropseudonorm\ in~\eqref{eq:pseudo_norm_l0},
  where $\TripleNormDualSphere$ in~\eqref{eq:triplenorm_Dual_unit_sphere}
  is the unit sphere of the 
  dual norm~$\TripleNormDual{\cdot}$,
  and where \( \SupportNorm{\TripleNormBall}{k} \) 
  in~\eqref{eq:generalized_k-support_norm_unit_ball} is the unit ball
  of the generalized $k$-support norm~\( \SupportNorm{\TripleNorm{\cdot}}{k} \).
\end{proposition}

\begin{proof}
  First, let us observe that the level set \( \LevelSet{\lzero}{k} \) 
  in~\eqref{eq:pseudonormlzero_level_set} is closed
  because the pseudonorm~$\lzero$ is lower semi continuous. Then, we get
  \begin{align*}
    \LevelSet{\lzero}{k} \cap \TripleNormDualSphere 
    &= 
      \closedconvexhull\bp{\LevelSet{\lzero}{k} \cap \TripleNormDualSphere} 
      \cap \TripleNormDualSphere 
      \tag{by Corollary~\ref{cor:convex_env} because
      \( \LevelSet{\lzero}{k} \cap \TripleNormDualSphere \subset \TripleNormDualSphere \) 
      and is closed, and because the unit ball~$\TripleNormDualBall$ is rotund }
    \\
    &= 
      \closedconvexhull\bp{ \bigcup_{ {\cardinal{K} \leq k}} 
      \np{ \FlatRR_{K} \cap \TripleNormDualSphere } }
      \cap \TripleNormDualSphere 
      \tag{as \( \LevelSet{\lzero}{k} = \bigcup_{\cardinal{K} \leq k} \FlatRR_{K} \)
      by~\eqref{eq:level_set_pseudonormlzero} }
    \\
    &= 
      \SupportNorm{\TripleNormBall}{k} \cap \TripleNormDualSphere 
  \end{align*}
  as 
  \( \closedconvexhull\bp{ \bigcup_{ {\cardinal{K} \leq k}} \np{ \FlatRR_{K} \cap \TripleNormDualSphere } }
  = \SupportNorm{\TripleNormBall}{k} \) 
  by~\eqref{eq:dual_support_norm_unit_ball} because the source norm~$\TripleNorm{\cdot}$ is orthant-monotonic.
  \medskip

  This ends the proof. 
\end{proof}

\bleue{%
  The result of
  Proposition~\ref{pr:level_set_pseudonormlzero_intersection_sphere_rotund}
  applies to the $\ell_p$-norm~$\norm{\cdot}_{p}$ for $p\in ]1,\infty[$. 
}

\begin{lemma}   
  \label{lemma:convex_env}
  Let $\TripleNorm{\cdot}$ be a norm on~$\RR^d$.
  Let $\subsetextr$ be a subset of 
  $\mathrm{extr}(\TripleNormBall) \subset \TripleNormSphere$, 
  the set of extreme points of~$\TripleNormBall$.
  If $A$ is a subset of $\subsetextr$, 
  then $A = \convexhull(A) \cap \subsetextr$. 
  If $A$ is a closed subset of $\subsetextr$,
  then $A = \closedconvexhull(A) \cap \subsetextr$.
\end{lemma}

\begin{proof} 
  We first prove that $A = \convexhull(A) \cap \subsetextr$ when $A \subset \subsetextr$. 
  Since $A \subset \convexhull(A)$ and $A\subset \subsetextr$, 
  we immediately get that $A \subset\convexhull(A) \cap \subsetextr$.
  To prove the reverse inclusion, we first start by proving that 
  $\convexhull(A)\cap \subsetextr \subset \mathrm{extr}\bp{\convexhull(A)}$,
  the set of extreme points of~$\convexhull(A)$.
  
  The proof is by contradiction.
  Suppose indeed that there exists $\primal\in \convexhull(A)\cap \subsetextr$ 
  and $x\not\in \mathrm{extr}\bp{\convexhull(A)}$. Then, by definition of an
  extreme point, we could find 
  $y \in \convexhull(A)$ and $z \in \convexhull(A)$, distinct from $\primal$,
  and such that $\primal = \lambda y + (1-\lambda) z$ for some $\lambda\in ]0,1[$.
  Notice that necessarily \( y \neq z \) (because, else, we would have
  $\primal=y=z$ which would contradict $y \neq \primal$ and $z \neq \primal$). 
  By assumption $A\subset \subsetextr$, 
  we deduce that $\convexhull(A) \subset \convexhull(\subsetextr)
  \subset \convexhull(\TripleNormSphere)= \TripleNormBall= 
  \defset{\primal \in \RR^d}{\TripleNorm{\primal} \leq 1}$, the unit ball, 
  and therefore that \( \TripleNorm{y} \leq 1 \) and \( \TripleNorm{z} \leq 1 \).
  If $y$ or $z$ were not in $\TripleNormSphere$ --- that is, if either
  \( \TripleNorm{y} < 1 \) or \( \TripleNorm{z} < 1 \) --- then we would obtain that
  \( \TripleNorm{\primal} \leq \lambda \TripleNorm{y} 
  + (1-\lambda) \TripleNorm{z} < 1 \) 
  since 
  $\lambda\in ]0,1[$;
  we would thus arrive at a contradiction 
  since $\primal$ could not be in the sphere~$\TripleNormSphere$  
  and thus not in $\subsetextr$. 
  Thus, both $y$ and $z$ must be in $\TripleNormSphere$, 
  and we have a contradiction. Indeed, by assumption
  that $\subsetextr$ is a subset of $\mathrm{extr}(\TripleNormSphere)$,
  no $\primal \in \subsetextr$ can be obtained as a convex combination of 
  $y \in \TripleNormSphere\backslash\{\primal\}$ 
  and $z \in \TripleNormSphere\backslash\{\primal\}$, with \( y \neq z \).

  Hence, we have proved by contradiction that 
  $\convexhull(A)\cap \subsetextr \subset \mathrm{extr}\bp{\convexhull(A)}$.
  We can conclude using the fact that 
  $\mathrm{extr}\bp{\convexhull(A)} \subset A$, because
  the convex closure operation cannot generate new extreme points,
  as proved in~\cite[Exercice 6.4]{hiriart1998optimisation}.
  \medskip

  Now, we consider the case where the subset $A$ of~$\subsetextr$
  is closed.  Using the first part of the proof we have that
  $A= \convexhull(A) \cap \subsetextr$.  
  Now, $A$ is closed by assumption and bounded
  since $A\subset \subsetextr \subset \TripleNormSphere$. 
  Thus, $A$ is a compact subset of~$\RR^d$ and, in a finite
  dimensional space, we get that 
  $\convexhull(A)$ is compact~\cite[Theorem~17.2]{Rockafellar:1970}, 
  thus closed. We conclude that 
  $A= \convexhull(A) \cap \subsetextr = 
  \overline{\convexhull(A)} \cap \subsetextr
  = \closedconvexhull(A) \cap \subsetextr$, 
  where the last equality comes 
  from~\cite[Prop.~3.46]{Bauschke-Combettes:2017}.

  This ends the proof. 
\end{proof}

If the unit ball~$\TripleNormBall$ is rotund, we then have that
$\TripleNormSphere=\mathrm{extr}(\TripleNormBall)$, and we can apply 
Lemma~\ref{lemma:convex_env} with $\subsetextr=\TripleNormSphere$ 
to obtain the following corollary.

\begin{corollary}  
  Let $\TripleNorm{\cdot}$ be a norm on~$\RR^d$.
  Suppose that the unit ball of the norm~$\TripleNorm{\cdot}$ is rotund.
  If $A$ is a subset of the unit sphere~$\TripleNormSphere$, 
  then $A = \convexhull(A) \cap \TripleNormSphere$. 
  If $A$ is a closed subset of $\TripleNormSphere$,
  then $A = \closedconvexhull(A) \cap \TripleNormSphere$.
  \label{cor:convex_env}
\end{corollary}

\subsection{Graded sequences of norms}
\label{Graded_sequences_of_norms}

In~\cite{Chancelier-DeLara:2022_CAPRA_OPTIMIZATION}, we introduced the notions
of (strictly) decreasingly graded sequences of norms.
In~\S\ref{Definitions_of_graded_sequences_of_norms},
we define (strictly) increasingly graded sequences of norms.
In~\S\ref{Sufficient_conditions_for_increasingly_graded_sequence_of_generalized_top-k_norms},
we display conditions for generalized top-$k$ norms to be 
(strictly) increasingly graded sequences.
In~\S\ref{Sufficient_conditions_for_decreasingly_graded_sequence_of_generalized_k-support_norms},
we display conditions for generalized $k$-support norms to be 
(strictly) decreasingly graded sequences.
In~\S\ref{Expressing_the_lzeropseudonorm_by_means_of_the_difference_between_two_norms},
we express the level sets of the \lzeropseudonorm\ in~\eqref{eq:pseudonormlzero_level_set}
by means of the difference between two norms.

\subsubsection{Definitions of graded sequences of norms}
\label{Definitions_of_graded_sequences_of_norms}

In a sense, a graded sequence of norms is a monotone sequence that detects 
the number of nonzero components of a vector in~$\RR^d$
when the sequence becomes stationary.

\begin{definition}
  We say that a sequence 
  \( \sequence{\TripleNorm{\cdot}_{k}}{k\in\ic{1,d}} \) of norms on~$\RR^d$
  is \emph{increasingly graded} 
  (resp. \emph{strictly increasingly graded})
  \wrt\ (with respect to) the \lzeropseudonorm\ if,
  for any \( \primal\in\RR^d \),
  one of the three following equivalent statements holds true.
  \begin{enumerate}
  \item 
    We have the implication (resp. equivalence), for any \( l\in\ic{1,d} \), 
    \begin{subequations}
      \begin{align}
        \lzero\np{\primal} = l 
        &\implies
          \TripleNorm{\primal}_{1} \leq \cdots \leq 
          \TripleNorm{\primal}_{l-1} \leq 
          \TripleNorm{{\primal}}_{l} = 
          \cdots =
          \TripleNorm{{\primal}}_{d} 
          \eqfinv
          \label{eq:increasingly_graded_a}
        \\
        \left(\right.\text{resp.} \qquad      
        \lzero\np{\primal} = l 
        &\iff 
          \TripleNorm{\primal}_{1} \leq \cdots \leq 
          \TripleNorm{\primal}_{l-1} <
          \TripleNorm{{\primal}}_{l} = 
          \cdots =
          \TripleNorm{{\primal}}_{d}
          \eqfinp 
          \left. \right)
          \label{eq:strictly_increasingly_graded_a}
      \end{align}
    \item 
      The sequence 
      \( k \in \ic{1,d} \mapsto \TripleNorm{\primal}_{k} \)
      is nondecreasing and we have the implication (resp. equivalence), for any \( l\in\ic{1,d} \), 
      \begin{align}
        \lzero\np{\primal} \leq l 
        & \implies
          \TripleNorm{{\primal}}_{l} = 
          \TripleNorm{{\primal}}_{d}
          \eqfinv
          \label{eq:increasingly_graded_b}
        \\
        \left(\right.\text{resp.} \qquad      
        \lzero\np{\primal} \leq l 
        & \iff 
          \TripleNorm{{\primal}}_{l} = 
          \TripleNorm{{\primal}}_{d}
          \quad \bp{       \iff 
          \TripleNorm{{\primal}}_{l} \leq 
          \TripleNorm{{\primal}}_{d} }
          \eqfinp
          \left. \right)
          \label{eq:strictly_increasingly_graded_b}
      \end{align}
    \item 
      The sequence 
      \( k \in \ic{1,d} \mapsto \TripleNorm{\primal}_{k} \)
      is nondecreasing and we have the inequality  (resp. equality)
      \begin{align}
        \lzero\np{\primal} 
        &\geq 
          \min \bset{k \in \ic{1,d} }%
          { \TripleNorm{{\primal}}_{k} = \TripleNorm{{\primal}}_{d} }
          \eqfinv 
          \label{eq:increasingly_graded_c}
        \\
        \left(\right.\text{resp.} \qquad      
        \lzero\np{\primal} 
        &= 
          \min \bset{k \in \ic{1,d} }%
          { \TripleNorm{{\primal}}_{k} = \TripleNorm{{\primal}}_{d} }
          \eqfinp
          \label{eq:strictly_increasingly_graded_c}
          \left. \right)
      \end{align}
      %
    \end{subequations}
  \end{enumerate}
  \label{de:increasingly_graded}
\end{definition}
These definitions of (strictly) increasingly graded mimic the ones of
(strictly) decreasingly graded in
\cite[Definition~1]{Chancelier-DeLara:2022_CAPRA_OPTIMIZATION}
(replace $\leq$ in~\eqref{eq:increasingly_graded_a} by~$\geq$,
replace $\leq$ and $<$ in~\eqref{eq:strictly_increasingly_graded_a} by~$\geq$
and $>$,
replace nondecreasing by nonincreasing in the two last items).

The property of orthant-strict monotonicity for norms, 
as introduced in Definition~\ref{de:orthant-strictly_monotonic}, 
proves especially relevant for the \lzeropseudonorm\
and sequences of generalized top-$k$ norms,
as the following Propositions~\ref{pr:increasingly_graded} 
and \ref{pr:decreasingly_graded_rotund} reveal.

\subsubsection{Sufficient conditions for increasingly graded sequence of generalized top-$k$ norms}
\label{Sufficient_conditions_for_increasingly_graded_sequence_of_generalized_top-k_norms}

We show that, when the source norm 
is orthant-(strictly) monotonic,
the sequence of induced generalized top-$k$ norms is (strictly) increasingly graded.

\begin{proposition}  
  \quad
  \begin{itemize}
  \item 
    If the norm $\TripleNorm{\cdot}$ is orthant-monotonic, 
    then the nondecreasing sequence 
    \( \bseqa{\TopNorm{\TripleNorm{\cdot}}{\LocalIndex}}{\LocalIndex\in\ic{1,d}} \)
    of generalized top-$k$ norms in~\eqref{eq:top_norm} 
    is increasingly graded with respect to the \lzeropseudonorm, that is,
    \begin{equation*}
      \lzero\np{\primal} \leq l 
      \Rightarrow 
      \TopNorm{\TripleNorm{\primal}}{l}=
      \TopNorm{\TripleNorm{\primal}}{d}
      \eqsepv \forall \primal \in \RR^d 
      \eqsepv \forall l\in\ic{0,d} 
      \eqfinp
    \end{equation*}
  \item 
    If the norm $\TripleNorm{\cdot}$ is orthant-strictly monotonic, 
    then the nondecreasing sequence 
    \( \bseqa{\TopNorm{\TripleNorm{\cdot}}{\LocalIndex}}{\LocalIndex\in\ic{1,d}} \)
    of generalized top-$k$ norms in~\eqref{eq:top_norm} 
    is strictly increasingly graded with respect to the \lzeropseudonorm, that is,
    \begin{equation*}
      \lzero\np{\primal} \leq l 
      \iff 
      \TopNorm{\TripleNorm{\primal}}{l}=
      \TopNorm{\TripleNorm{\primal}}{d}
      \eqsepv \forall \primal \in \RR^d 
      \eqsepv \forall l\in\ic{0,d} 
      \eqfinp
    \end{equation*}
  \end{itemize}
  \label{pr:increasingly_graded}
\end{proposition}

\begin{proof} 

  \noindent $\bullet$
  We suppose that the norm $\TripleNorm{\cdot}$ is orthant-monotonic.
  As the sequence 
  \( \bseqa{\TopNorm{\TripleNorm{\cdot}}{\LocalIndex}}{\LocalIndex\in\ic{1,d}} \)
  of generalized top-$k$ norms in~\eqref{eq:top_norm} 
  is nondecreasing
  by the inequalities~\eqref{eq:generalized_top-k_norm_inequalities},
  it suffices to show~\eqref{eq:increasingly_graded_b} ---
  that is, \( \lzero\np{\primal} \leq l 
  \Rightarrow
  \TopNorm{\TripleNorm{\primal}}{d}
  = \TopNorm{\TripleNorm{\primal}}{l} \) ---
  to prove that the sequence
  is increasingly graded with respect to the \lzeropseudonorm.

  For this purpose, we consider \( \primal\in\RR^d \),
  we put \( L=\Support{\primal} \) and we suppose that 
  $\lzero(\primal)=\cardinal{L} \leq l $.
  We now show that \( \TopNorm{\TripleNorm{\primal}}{d}
  = \TopNorm{\TripleNorm{\primal}}{l} \).
  Since \( \primal=\primal_L \), we have 
  \( \TripleNorm{\primal} = \TripleNorm{\primal_L} = \TripleNorm{\primal_L}_L
  \leq \TopNorm{\TripleNorm{\primal}}{l} \), by the very definition~\eqref{eq:top_norm}
  of the generalized top-$l$ norm \( \TopNorm{\TripleNorm{\cdot}}{l} \).
  On the one hand, we have just obtained that 
  \( \TripleNorm{\primal} \leq \TopNorm{\TripleNorm{\primal}}{l} \).
  On the other hand, we have that 
  \(       \TopNorm{\TripleNorm{\primal}}{l} \leq 
  \TopNorm{\TripleNorm{\primal}}{l+1} \leq \cdots \leq 
  \TopNorm{\TripleNorm{\primal}}{d} =  \TripleNorm{\primal} \)
  by the inequalities~\eqref{eq:generalized_top-k_norm_inequalities}
  and the last equality comes from Item~\ref{it:generalized_top-d_norm_equality}
  in Proposition~\ref{pr:source_norm_orthant-monotonic_generalized_top-k_norm}
  since the norm~$\TripleNorm{\cdot}$ is orthant-monotonic.
  Hence, we deduce that 
  \( \TripleNorm{\primal} = \TopNorm{\TripleNorm{\primal}}{d}
  = \cdots = \TopNorm{\TripleNorm{\primal}}{l} \), so that 
  $\TopNorm{\TripleNorm{\primal}}{k}$ is stationary for $k\ge l$.
  \medskip

  \noindent $\bullet$
  We suppose that the norm $\TripleNorm{\cdot}$ is orthant-strictly monotonic.
  To prove that the equivalence~\eqref{eq:strictly_increasingly_graded_a} holds true
  for the sequence 
  \( \bseqa{\TopNorm{\TripleNorm{\cdot}}{\LocalIndex}}{\LocalIndex\in\ic{1,d}} \),
  it is easily seen that it suffices to show that
  \begin{equation}
    \lzero\np{\primal} = l \Rightarrow
    \TopNorm{\TripleNorm{\primal}}{1} < \cdots
    < \TopNorm{\TripleNorm{\primal}}{l-1} < \TopNorm{\TripleNorm{\primal}}{l} 
    =  \TopNorm{\TripleNorm{\primal}}{l+1}
    = \cdots = \TopNorm{\TripleNorm{\primal}}{d} 
    \eqsepv \forall \primal\in\RR^d 
    \eqfinp
  \end{equation}
  We consider \( \primal\in\RR^d \).
  We put \( L=\Support{\primal} \) and we suppose that 
  $\lzero(\primal)=\cardinal{L} = l $.
  As the norm $\TripleNorm{\cdot}$ is orthant-strictly monotonic,
  it is orthant-monotonic, so that the equalities 
  \( \TopNorm{\TripleNorm{\primal}}{l} 
  =  \TopNorm{\TripleNorm{\primal}}{l+1}
  = \cdots = \TopNorm{\TripleNorm{\primal}}{d} \) above hold true
  (as just established in the first part of the proof).
  Therefore, it only remains to prove that 
  \( \TopNorm{\TripleNorm{\primal}}{1} < \cdots
  < \TopNorm{\TripleNorm{\primal}}{l-1} < \TopNorm{\TripleNorm{\primal}}{l} \).

  There is nothing to show for~$l=0$.
  Now, for \( l \geq 1 \) and for any \( k \in \ic{0,l-1} \), we have
  \begin{align*}
    \TopNorm{\TripleNorm{\primal}}{k}  
    &=
      \sup_{\cardinal{K} \leq k} 
      \TripleNorm{\primal_{K}}
      \tag{by definition~\eqref{eq:top_norm} of the generalized top-$k$ norm } 
    \\
    &=
      \sup_{\cardinal{K} \leq k} 
      \TripleNorm{\primal_{K \cap L}}
      \tag{because \( \primal_L=\primal \) by definition of the set~$L=\Support{\primal}$ }
    \\
    &=
      \sup_{\cardinal{K'} \leq k, K' \subset L} 
      \TripleNorm{\primal_{K'}}
      \tag{by setting \( K'=K \cap L \) }
    \\
    &=
      \sup_{\cardinal{K} \leq k, K \subset L} 
      \TripleNormDual{\dual_{K}}
      \tag{the same but with $K$ instead of $K'$}
    \\
    &=
      \sup_{\cardinal{K} \leq k, K \subsetneq L } 
      \TripleNorm{\primal_{K}}
      \tag{because \( \cardinal{K} \leq k \leq l-1 < l =\cardinal{L} \)
      implies that \( K \neq L \)}
    \\
    &< 
      \sup_{\substack{\cardinal{K} \leq k, \LocalIndex \in L\setminus K \\ K \subsetneq L} } 
    \TripleNorm{\primal_{K \cup \na{\LocalIndex}}}
    \intertext{%
    because the set \( L \setminus K \) is nonempty
    (having cardinality 
    \( \cardinal{L}-\cardinal{K}=l-\cardinal{K} \geq k+1-\cardinal{K} \geq 1 \)),
    and because, since the norm $\TripleNorm{\cdot}$ is orthant-strictly monotonic, 
    using Item~\ref{it:SICS} in Proposition~\ref{pr:orthant-strictly_monotonic},
    we obtain that \( \TripleNorm{\primal_{K}} < \TripleNorm{\primal_{K \cup \na{\LocalIndex}}} \) 
    as \( \primal_{K} \neq \primal_{K \cup \na{\LocalIndex}} \) for at least one
    \( \LocalIndex \in L\setminus K \) since  \( L=\Support{\primal} \)}
    & \leq
      \sup_{\cardinal{J} \leq k+1, J \subset L } 
      \TripleNorm{\primal_{J}}
      \tag{as all the subsets $K'=K \cup \na{\LocalIndex}$ are such that $K' \subset L $ and 
      \( \cardinal{K'}=k+1 \)}
    \\
    & \leq 
      \TopNorm{\TripleNorm{\primal}}{k+1}  
  \end{align*}
  by definition~\eqref{eq:top_norm} of the generalized top-$k+1$ norm
  (in fact the last inequality is easily shown to be an equality as
  \( \primal_L=\primal \)).
  Thus, for any \( k \in \ic{0,l-1} \), we have established that 
  \( \TopNorm{\TripleNorm{\primal}}{k} < \TopNorm{\TripleNorm{\primal}}{k+1} \). 
  \medskip

  This ends the proof.
\end{proof}

We show that, when the source norm is orthant-strictly monotonic,
it is equivalent either that 
the sequence of induced generalized top-$k$ norms be strictly increasingly graded.
or that the dual norm~$\TripleNormDual{\cdot}$ be orthant-strictly monotonic.

\begin{proposition}
  The following statements are equivalent.
  \begin{enumerate}
  \item 
    The dual norm $\TripleNormDual{\cdot}$ is orthant-strictly monotonic
    and the sequence 
    \( \bseqa{\TopNorm{\TripleNorm{\cdot}}{\LocalIndex}}{\LocalIndex\in\ic{1,d}} \)
    of generalized top-$k$ norms  in~\eqref{eq:top_norm} 
    is strictly increasingly graded with respect to the \lzeropseudonorm.
    \label{it:calL0_1}
  \item 
    Both the norm $\TripleNorm{\cdot}$ 
    and the dual norm $\TripleNormDual{\cdot}$
    are orthant-strictly monotonic.
    \label{it:calL0_2}
  \end{enumerate}
\end{proposition}

\begin{proof}

  \noindent $\bullet$ 
  Suppose that Item~\ref{it:calL0_1} is satisfied
  and let us show that Item~\ref{it:calL0_2} holds true.
  For this, it suffices to prove that the norm $\TripleNorm{\cdot}$
  is orthant-strictly monotonic.   
  To prove that the norm $\TripleNorm{\cdot}$
  is orthant-strictly monotonic, we will show that 
  Item~\ref{it:SICS} in Proposition~\ref{pr:orthant-strictly_monotonic}
  holds true for~$\TripleNorm{\cdot}$. 
  For this purpose, we consider \( \primal \in \RR^d \)
  and \( J \subsetneq K \subset\ic{1,d} \) such that 
  $ \primal_J \neq \primal_K $. 
  By definition of the \lzeropseudonorm\ in~\eqref{eq:pseudo_norm_l0}, 
  we have \( j=\lzero\np{\primal_J} < k=\lzero\np{\primal_K} \).

  On the one hand, as the dual norm $\TripleNormDual{\cdot}$ is orthant-strictly monotonic,
  it is orthant-monotonic, so that 
  the norm $\TripleNorm{\cdot}$ is also orthant-monotonic,
  as proved in~\cite[Theorem~ 2.23]{Gries:1967} 
  (equivalence between Item~\ref{it:orthant-monotonic_IN_pr:orthant-monotonic}
  and Item~\ref{it:dual_orthant-monotonic_IN_pr:orthant-monotonic} 
  in Proposition~\ref{pr:orthant-monotonic}).
  As a consequence, so are the norms in the sequence 
  \( \bseqa{\TopNorm{\TripleNorm{\cdot}}{\LocalIndex}}{\LocalIndex\in\ic{1,d}} \)
  by Item~\ref{it:generalized_top-ksupport_norm_orthant-monotonic}
  in Proposition~\ref{pr:source_norm_orthant-monotonic_generalized_top-k_norm}, 
  and we get that 
  \( \TopNorm{\TripleNorm{\primal_J}}{k-1} 
  \leq \TopNorm{\TripleNorm{\primal_K}}{k-1} \), in particular,
  by the equivalence between
  Item~\ref{it:orthant-monotonic_IN_pr:orthant-monotonic} and Item~\ref{it:ICS} in
  Proposition~\ref{pr:orthant-monotonic}.

  On the other hand, since, by assumption, the sequence 
  \( \bseqa{\TopNorm{\TripleNorm{\cdot}}{\LocalIndex}}{\LocalIndex\in\ic{1,d}} \)
  of generalized top-$k$ norms is strictly increasingly graded 
  with respect to the \lzeropseudonorm,
  we have by~\eqref{eq:strictly_increasingly_graded_a} that,
  on the one hand, 
  \( \TopNorm{\TripleNorm{\primal_J}}{1} \leq \cdots \leq 
  \TopNorm{\TripleNorm{\primal_J}}{j-1}  < 
  \TopNorm{\TripleNorm{\primal_J}}{j} =
  \cdots = \TopNorm{\TripleNorm{\primal_J}}{d} =\TripleNorm{\primal_J} \),
  because \( j=\lzero\np{\primal_J} \), 
  and, on the other hand, 
  \( \TopNorm{\TripleNorm{\primal}}{1} \leq \cdots \leq 
  \TopNorm{\TripleNorm{\primal_K}}{k-1}  < 
  \TopNorm{\TripleNorm{\primal_K}}{k} =
  \cdots = 
  \TopNorm{\TripleNorm{\primal_K}}{d} =\TripleNorm{\primal_K}\),
  because \( k=\lzero\np{\primal_K} \).
  Since $j < k$, we deduce that 
  \[
    \TripleNorm{\primal_J} 
    = \TopNorm{\TripleNorm{\primal_J}}{j} 
    = \TopNorm{\TripleNorm{\primal_J}}{k-1} 
    \leq \TopNorm{\TripleNorm{\primal_K}}{k-1} 
    < \TopNorm{\TripleNorm{\primal_K}}{k} 
    = \TripleNorm{\primal_K}
    \eqfinv
  \]
  and therefore that \( \TripleNorm{\primal_J} < \TripleNorm{\primal_K} \).
  Thus, Item~\ref{it:SICS} in Proposition~\ref{pr:orthant-strictly_monotonic}
  holds true for~$\TripleNorm{\cdot}$, so that the norm $\TripleNorm{\cdot}$
  is orthant-strictly monotonic.
  Hence, we have shown that Item~\ref{it:calL0_2} is satisfied.
  \medskip

  \noindent $\bullet$ 
  Suppose that Item~\ref{it:calL0_2} is satisfied
  and let us show that Item~\ref{it:calL0_1} holds true. 

  Since the norm $\TripleNorm{\cdot}$ is orthant-strictly monotonic,
  it has been proved in Proposition~\ref{pr:increasingly_graded}
  that the sequence \( \bseqa{\TopNorm{\TripleNorm{\cdot}}{\LocalIndex}}{\LocalIndex\in\ic{1,d}} \)
  is strictly increasingly graded with respect to the \lzeropseudonorm.
  Hence, Item~\ref{it:calL0_1} holds true.
  \medskip

  This ends the proof.
\end{proof}

\subsubsection{Sufficient conditions for decreasingly graded sequence of generalized $k$-support norms}
\label{Sufficient_conditions_for_decreasingly_graded_sequence_of_generalized_k-support_norms}

There is an asymetry in that 
the property of orthant-strict monotonicity for norms
does not proves especially relevant for the \lzeropseudonorm\
and sequences of $k$-support norms.
Indeed, consider the source norm~$\TripleNorm{\cdot}=\norm{\cdot}_{1}$,
that is, the $\ell_1$ norm which is orthant-strict monotonic.
By Table~\ref{tab:Examples} (third column), we know that the 
$k$-support norms are the norms \( \SupportNorm{\TripleNorm{\cdot}}{k} =
\LpSupportNorm{\cdot}{\infty}{k} =
\max \na{ \Norm{\cdot}_{1} / k , \Norm{\cdot}_{\infty} } \), for $k\in\ic{1,d}$.
Now, the nonincreasing sequence 
\( \bseqa{\SupportNorm{\TripleNorm{\cdot}}{\LocalIndex}}{\LocalIndex\in\ic{1,d}} \)
of norms is not strictly decreasingly graded with respect to the \lzeropseudonorm\
when $d \geq 2$.
Indeed, for any \( \varepsilon \in ]0,1[ \), 
the vector \( \dual=\bp{ \varepsilon/(d-1), \ldots, \varepsilon/(d-1), 1 } \)
is such that
\[
  \lzero\np{\dual}=d \text{ and } 
  \SupportNorm{\TripleNorm{\dual}}{1} > \SupportNorm{\TripleNorm{\dual}}{2} = \cdots
  = \SupportNorm{\TripleNorm{\dual}}{d} 
\]
because \( \SupportNorm{\TripleNorm{\dual}}{k} =
\max \ba{ \Norm{\dual}_{1}/k , \Norm{\dual}_{\infty} }
=  \max \ba{ (\varepsilon+1)/k , 1 }
\), for \( k\in\ic{1,d} \), so that
\( \varepsilon+1=\SupportNorm{\TripleNorm{\dual}}{1} > \SupportNorm{\TripleNorm{\dual}}{2} = \cdots
= \SupportNorm{\TripleNorm{\dual}}{d}=1 \).
However, we establish the following result.

\begin{proposition}
  \quad
  \begin{itemize}
  \item 
    If the source norm~$\TripleNorm{\cdot}$ is orthant-monotonic,
    then the nonincreasing sequence 
    \( \bseqa{\SupportNorm{\TripleNorm{\cdot}}{\LocalIndex}}{\LocalIndex\in\ic{1,d}} \)
    of generalized $k$-support norms in~\eqref{eq:support_norm} 
    is decreasingly graded with respect to the \lzeropseudonorm,
    that is,
    \begin{equation}
      \lzero\np{\dual} \leq l 
      \Rightarrow 
      \SupportNorm{\TripleNorm{\dual}}{l}=
      \SupportNorm{\TripleNorm{\dual}}{d}
      \eqsepv \forall \dual \in \RR^d 
      \eqsepv \forall l\in\ic{0,d} 
      \eqfinp
      \label{eq:OM_SN_DG}
    \end{equation}
  \item 
    If the source norm~$\TripleNorm{\cdot}$ is orthant-monotonic,
    and if the normed space 
    \( \bp{\RR^d,\TripleNormDual{\cdot}} \) is strictly convex,
    then the nonincreasing sequence 
    \( \bseqa{\SupportNorm{\TripleNorm{\cdot}}{\LocalIndex}}{\LocalIndex\in\ic{1,d}} \)
    of generalized $k$-support norms in~\eqref{eq:support_norm} 
    is strictly decreasingly graded with respect to the \lzeropseudonorm,
    that is,
    \begin{equation}
      \lzero\np{\dual} \leq l 
      \iff
      \SupportNorm{\TripleNorm{\dual}}{l}=
      \SupportNorm{\TripleNorm{\dual}}{d}
      \eqsepv \forall \dual \in \RR^d 
      \eqsepv \forall l\in\ic{0,d} 
      \eqfinp
      \label{eq:OM_rotund_SN_SDG}
    \end{equation}
  \end{itemize}
  \label{pr:decreasingly_graded_rotund}
\end{proposition}

\begin{proof}
  A direct proof would use~\cite[Proposition~6]{Chancelier-DeLara:2022_CAPRA_OPTIMIZATION}
  with \( \TripleNormDual{\cdot} \) as source norm,
  and the property that \( \SupportNorm{\TripleNorm{\cdot}}{\LocalIndex} \)
  coincides with the coordinate-$k$ norm~\cite[Definition~3]{Chancelier-DeLara:2022_CAPRA_OPTIMIZATION}
  induced by~\( \TripleNormDual{\cdot} \) when the norm~$\TripleNorm{\cdot}$ is orthant-monotonic.
  We give a self-contained proof for the sake of completeness.
  \medskip

  \noindent$\bullet$
  We suppose that the source norm~$\TripleNorm{\cdot}$ is orthant-monotonic.

  For any \( \dual \in \RR^d \) and
  for any \( k \in \ic{1,d} \), we have\footnote{%
    In what follows, by ``or'', we mean the so-called \emph{exclusive or}
    (exclusive disjunction). Thus, every ``or'' should be understood as ``or
    $\dual\not=0$ and''.
    \label{ft:exclusive_or}}
  \begin{align*}
    \dual \in \LevelSet{\lzero}{k} 
    &\Leftrightarrow
      \dual=0 \text{ or }
      \frac{\dual}{\TripleNormDual{\dual}} \in \LevelSet{\lzero}{k}
      \tag{by 0-homogeneity~\eqref{eq:lzeropseudonorm_is_0-homogeneous}
      of the \lzeropseudonorm, and
      by definition~\eqref{eq:pseudonormlzero_level_set}
      of $\LevelSet{\lzero}{k}$ }
    \\
    &\Leftrightarrow 
      \dual=0 \text{ or }
      \frac{\dual}{\TripleNormDual{\dual}} \in 
      \LevelSet{\lzero}{k} \cap \TripleNormDualSphere 
      \tag{as \( \frac{\dual}{\TripleNormDual{\dual}} \in \TripleNormDualSphere \)}
    \\
    &\Leftrightarrow 
      \dual=0 \text{ or }
      \frac{\dual}{\TripleNormDual{\dual}} \in 
      \bigcup_{\cardinal{K} \leq k} \np{ \FlatRR_{K} \cap \TripleNormDualSphere
      }
      \tag{as \( \LevelSet{\lzero}{k} = \bigcup_{\cardinal{K} \leq k} \FlatRR_{K} \)
      by~\eqref{eq:level_set_pseudonormlzero} }
    \\
    &\Rightarrow
      \dual=0 \text{ or }
      \frac{\dual}{\TripleNormDual{\dual}} \in 
      \closedconvexhull\bp{ \bigcup_{ {\cardinal{K} \leq k}} 
      \np{ \FlatRR_{K} \cap \TripleNormDualSphere } }
      \tag{as \( S \subset \closedconvexhull\np{S} \) for any subset~$S$ of $\RR^d$}
    \\
    &\Rightarrow
      \dual=0 \text{ or }
      \frac{\dual}{\TripleNormDual{\dual}} \in 
      \SupportNorm{\TripleNormBall}{k} 
      \tag{as 
      \( \closedconvexhull\bp{ \bigcup_{ {\cardinal{K} \leq k}} \np{ \FlatRR_{K} \cap \TripleNormDualSphere } }
      = \SupportNorm{\TripleNormBall}{k} \) 
      by~\eqref{eq:dual_support_norm_unit_ball} because the source norm~$\TripleNorm{\cdot}$ is orthant-monotonic}
    \\
    &\Rightarrow
      \dual=0 \text{ or }
      \SupportNorm{{\TripleNorm{\frac{\dual}{\TripleNormDual{\dual}}}}}{k}
      \leq 1 
      \tag{by definition~\eqref{eq:generalized_k-support_norm_unit_ball}
      of the unit ball~\( \SupportNorm{\TripleNormBall}{k} \) } 
    \\
    &\Rightarrow
      \SupportNorm{\TripleNorm{\dual}}{k}
      \leq \TripleNormDual{\dual}
      =\SupportNorm{\TripleNorm{\dual}}{d} 
      \tag*{(where the last equality comes
      from Item~\ref{it:generalized_top-d_norm_equality}}
    \\
    \tag*{in Proposition~\ref{pr:source_norm_orthant-monotonic_generalized_top-k_norm}
    since the norm~$\TripleNorm{\cdot}$ is orthant-monotonic)}
    \\
    &\Rightarrow
      \SupportNorm{\TripleNorm{\dual}}{k}
      = \SupportNorm{\TripleNorm{\dual}}{d} 
      \tag{as \( \SupportNorm{\TripleNorm{\dual}}{k}
      \geq \SupportNorm{\TripleNorm{\dual}}{d} \) by~\eqref{eq:generalized_k-support_norm_inequalities} }
      \eqfinp
  \end{align*}
  Therefore, we have obtained~\eqref{eq:OM_SN_DG}. 
  As the sequence 
  \( \bseqa{\SupportNorm{\TripleNorm{\cdot}}{\LocalIndex}}{\LocalIndex\in\ic{1,d}} \)
  of generalized $k$-support norms is nonincreasing
  by~\eqref{eq:generalized_k-support_norm_inequalities}, 
  we conclude that it
  is decreasingly graded with respect to the \lzeropseudonorm\
  (see the comments after Definition~\ref{de:increasingly_graded}).
  \medskip

  \noindent$\bullet$
  We suppose that the source norm~$\TripleNorm{\cdot}$ is orthant-monotonic
  and that the normed space 
  \( \bp{\RR^d,\TripleNormDual{\cdot}} \) is strictly convex.

  For any \( \dual \in \RR^d \) and
  for any \( k \in \ic{1,d} \), we have\footnote{%
    See Footnote~\ref{ft:exclusive_or}.}
  \begin{align*}
    \dual \in \LevelSet{\lzero}{k} 
    &\Leftrightarrow 
      \dual=0 \text{ or }
      \frac{\dual}{\TripleNormDual{\dual}} \in \LevelSet{\lzero}{k}
      \tag{by 0-homogeneity~\eqref{eq:lzeropseudonorm_is_0-homogeneous}
      of the \lzeropseudonorm, and
      by definition~\eqref{eq:pseudonormlzero_level_set}
      of $\LevelSet{\lzero}{k}$ }
    \\
    &\Leftrightarrow 
      \dual=0 \text{ or }
      \frac{\dual}{\TripleNormDual{\dual}} \in 
      \LevelSet{\lzero}{k} \cap \TripleNormDualSphere 
      \tag{as \( \frac{\dual}{\TripleNormDual{\dual}} \in \TripleNormDualSphere \)}
    \\
    &\Leftrightarrow 
      \dual=0 \text{ or }
      \frac{\dual}{\TripleNormDual{\dual}} \in 
      \SupportNorm{\TripleNormBall}{k} \cap \TripleNormDualSphere 
      \intertext{by~\eqref{eq:level_set_l0_inter_sphere_b} 
      since the assumptions of
      Proposition~\ref{pr:level_set_pseudonormlzero_intersection_sphere_rotund}
      --- namely, the source norm~$\TripleNorm{\cdot}$ is orthant-monotonic
      and the normed space 
      \( \bp{\RR^d,\TripleNormDual{\cdot}} \) is strictly convex ---
      are satisfied }
    & \Leftrightarrow 
      \dual=0 \text{ or }
      \frac{\dual}{\TripleNormDual{\dual}} \in 
      \SupportNorm{\TripleNormBall}{k}
      \tag{as $\frac{\dual}{\TripleNormDual{\dual}} \in \TripleNormDualSphere$}
    \\
    &\Leftrightarrow 
      \dual=0 \text{ or }
      \SupportNorm{\TripleNorm{\frac{\dual}{\TripleNormDual{\dual}}}}{k}
      \leq 1 
      \tag{by definition~\eqref{eq:generalized_k-support_norm_unit_ball}
      of the unit ball~\( \SupportNorm{\TripleNormBall}{k} \) } 
    \\
    &\Leftrightarrow 
      \SupportNorm{\TripleNorm{\dual}}{k}
      \leq \TripleNormDual{\dual}
      =\SupportNorm{\TripleNorm{\dual}}{d} 
      \tag*{(where the last equality comes
      from Item~\ref{it:generalized_top-d_norm_equality}}
    \\
    \tag*{ in Proposition~\ref{pr:source_norm_orthant-monotonic_generalized_top-k_norm}
    since the norm~$\TripleNorm{\cdot}$ is orthant-monotonic)}
    \\
    &\Leftrightarrow 
      \SupportNorm{\TripleNorm{\dual}}{k}
      = \SupportNorm{\TripleNorm{\dual}}{d} 
      \tag{as \( \SupportNorm{\TripleNorm{\dual}}{k}
      \geq \SupportNorm{\TripleNorm{\dual}}{d} \) by~\eqref{eq:generalized_k-support_norm_inequalities} }
      \eqfinp
  \end{align*}
  Therefore, we have obtained~\eqref{eq:OM_rotund_SN_SDG}. 
  As the sequence 
  \( \bseqa{\SupportNorm{\TripleNorm{\cdot}}{\LocalIndex}}{\LocalIndex\in\ic{1,d}} \)
  of generalized $k$-support norms is nonincreasing
  by~\eqref{eq:generalized_k-support_norm_inequalities}, 
  we conclude that it
  is strictly decreasingly graded with respect to the \lzeropseudonorm\
  (see the comments after Definition~\ref{de:increasingly_graded}).
  \medskip

  This ends the proof.
\end{proof}

\subsubsection{Expressing the \lzeropseudonorm\ 
  by means of the difference between two norms} 
\label{Expressing_the_lzeropseudonorm_by_means_of_the_difference_between_two_norms}

Propositions~\ref{pr:increasingly_graded} and \ref{pr:decreasingly_graded_rotund}
open the way for so-called ``difference of convex'' (DC)
optimization methods~\cite{Tono-Takeda-Gotoh:2017} 
to achieve sparsity.

Indeed, if the source norm $\TripleNorm{\cdot}$ is orthant-strictly monotonic, 
the level sets of the \lzeropseudonorm\ in~\eqref{eq:pseudonormlzero_level_set}
can be expressed
by means of the difference between two norms (one being a generalized top-$k$
norm), as follows,
\begin{subequations}
  \begin{equation}
    \LevelSet{\lzero}{k} =   
    \defset{\primal \in \RR^d}{%
      \TripleNorm{\primal}=\TopNorm{\TripleNorm{\primal}}{k} }
    =   
    \defset{\primal \in \RR^d}{%
      \TripleNorm{\primal} \leq \TopNorm{\TripleNorm{\primal}}{k} }
    \eqsepv \forall k\in\ic{0,d} 
    \eqfinv
  \end{equation}
  and the \lzeropseudonorm\ has the expression
  (see~\eqref{eq:strictly_increasingly_graded_c})
  \begin{equation}
    \lzero\np{\primal} 
    = \min \Bset{k \in \ic{1,d} }%
    { \TopNorm{\TripleNorm{\primal}}{k} = \TripleNorm{\primal} }
    \eqsepv \forall \primal \in \RR^d
    \eqfinp
  \end{equation}
\end{subequations}
As the $\ell_p$-norm~$\norm{\cdot}_{p}$ 
and its dual norm are orthant-strictly monotonic for $p\in ]1,\infty[$, 
the formulas above hold true with the
\lptopnorm{p}{k}
\( \TopNorm{\TripleNorm{\cdot}}{k} =
\LpTopNorm{\cdot}{p}{k} \) (see second column of Table~\ref{tab:Examples}).
\medskip

If the source norm~$\TripleNorm{\cdot}$ is orthant-monotonic
and the normed space 
\( \bp{\RR^d,\TripleNormDual{\cdot}} \) is strictly convex,
the level sets of the \lzeropseudonorm\ in~\eqref{eq:pseudonormlzero_level_set}
can be expressed
by means of the difference between two norms (one being a generalized $k$-support norm), as follows,
\begin{subequations}
  \begin{equation}
    \LevelSet{\lzero}{k} =   
    \defset{\dual \in \RR^d}{%
      \SupportNorm{\TripleNorm{\dual}}{k}
      = \TripleNormDual{\dual} }
    =   
    \defset{\dual \in \RR^d}{%
      \SupportNorm{\TripleNorm{\dual}}{k}
      \leq \TripleNormDual{\dual} }
    \eqsepv \forall k\in\ic{0,d} 
    \eqfinv
  \end{equation} 
  and the \lzeropseudonorm\ has the expression
  (see~\eqref{eq:strictly_increasingly_graded_c})
  \begin{equation}
    \lzero\np{\dual} 
    = \min \Bset{k \in \ic{1,d} }%
    { \SupportNorm{\TripleNorm{\dual}}{k} = \TripleNormDual{\dual} }
    \eqsepv \forall \dual \in \RR^d
    \eqfinp
  \end{equation}
\end{subequations}
As the $\ell_p$-norm~$\norm{\cdot}_{p}$ is orthant-monotonic
and the normed space 
\( \bp{\RR^d,\norm{\cdot}_{q}} \) is strictly convex,
when $p\in ]1,\infty[$ and $1/p+1/q=1$, 
the formulas above hold true with the \lpsupportnorm{q}{k}
\( \SupportNorm{\TripleNorm{\cdot}}{k} =\LpSupportNorm{\dual}{q}{k} \)
for $q\in ]1,\infty[$ (see Table~\ref{tab:Examples}).

\section{Conclusion}

In sparse optimization problems, one looks for solution that have few
nonzero components, that is, sparsity is exactly measured by the \lzeropseudonorm.
However, the mathematical expression of the \lzeropseudonorm,
taking integer values, 
makes it difficult to handle it in optimization problems.
To overcome this difficulty, one can try to replace the embarrassing 
\lzeropseudonorm\ by nicer terms, like norms.
In this paper, we contribute to this program by 
bringing up three new concepts for norms, 
and show how they prove especially relevant for the \lzeropseudonorm.

First, we have introduced a new class of 
orthant-strictly monotonic norms, 
inspired from orthant-monotonic norms.
With such a norm, when one component of a vector moves away from zero,
the norm of the vector strictly grows. 
Thus, an orthant-strictly monotonic norm is sensitive to the support 
of a vector, like the \lzeropseudonorm.
We have provided different characterizations of orthant-strictly monotonic norms
(and added a new characterization of orthant-monotonic norms).
Second, we have extended already known concepts 
of top-$k$ and $k$-support norms to 
sequences of generalized top-$k$ and $k$-support norms,
generated from any source norm (and not only from the $\ell_p$~norms), and have studied their properties. 
Third, we have introduced the notion of sequences of norms
that are strictly increasingly graded with respect to the \lzeropseudonorm.
A graded sequence detects the number of nonzero components of a vector
when the sequence becomes stationary.

With these three notions, we have proved that, 
when the source norm 
is orthant-strictly monotonic,
the sequence of induced generalized top-$k$ norms is strictly increasingly graded.
We have also shown that, when the source norm 
is orthant-monotonic and 
that the normed space~$\RR^d$ is strictly convex when equipped with the dual norm,
the sequence of induced generalized $k$-support norms
is strictly decreasingly graded.

\begin{table}[htbp]
  \centering
  \begin{tabular}{||l||c|c||c||c|c||c|c||}
    \hline \hline
    & \multicolumn{3}{c||}{ \( \sequence{\TopNorm{\TripleNorm{\cdot}}{\LocalIndex}}{\LocalIndex\in\ic{1,d}} \) }  
    & \multicolumn{4}{c||}{ \( \sequence{\SupportNorm{\TripleNorm{\cdot}}{\LocalIndex}}{\LocalIndex\in\ic{1,d}} \) } 
    \\
    & \multicolumn{3}{c||}{ increasingly }
    & \multicolumn{4}{c||}{ decreasingly }
    \\
    & \multicolumn{2}{c||}{graded} & strictly graded &
                                                       \multicolumn{2}{c||}{graded} & \multicolumn{2}{c||}{strictly graded}
    \\ 
    \hline \hline
    \( \TripleNorm{\cdot} \) is orthant-monotonic &\checkmark &&&\checkmark&&\checkmark&
    \\
    \hline
    \( \TripleNorm{\cdot} \) is orthant-strictly monotonic &&& \checkmark &&&&
    \\
    \hline
    \( \TripleNormDual{\cdot} \) is orthant-monotonic&&\checkmark &&& \checkmark &&\checkmark
    \\
    \hline
    \( \np{\RR^d,\TripleNormDual{\cdot}} \) 
    is strictly convex &&&&&& \checkmark &\checkmark
    \\
    \hline \hline
  \end{tabular}
  \caption{Table of results. It reads by columns as follows:
    to obtain that \( \sequence{\TopNorm{\TripleNorm{\cdot}}{\LocalIndex}}{\LocalIndex\in\ic{1,d}} \) 
    is increasingly strictly graded (column~4), it suffices that 
    \( \TripleNorm{\cdot} \) be orthant-strictly monotonic
    (the only checkmark~\checkmark in column~4);
    to obtain that \( \sequence{\TopNorm{\TripleNorm{\cdot}}{\LocalIndex}}{\LocalIndex\in\ic{1,d}} \) 
    is increasingly graded (columns~2 and 3), it suffices that
    either \( \TripleNorm{\cdot} \) be orthant-monotonic
    (the only checkmark~\checkmark in column~2)
    or \( \TripleNormDual{\cdot} \) be orthant-monotonic
    (the only checkmark~\checkmark in column~3);
    to obtain that \( \sequence{\SupportNorm{\TripleNorm{\cdot}}{\LocalIndex}}{\LocalIndex\in\ic{1,d}} \) 
    is decreasingly strictly graded (columns~7 and 8), it suffices 
    either that \( \TripleNorm{\cdot} \) be orthant-monotonic
    and that \( \np{\RR^d,\TripleNormDual{\cdot}} \) be strictly convex
    (two checkmarks~\checkmark in column~7)
    or that \( \TripleNormDual{\cdot} \) be orthant-monotonic
    and that \( \np{\RR^d,\TripleNormDual{\cdot}} \) be strictly convex
    (two checkmarks~\checkmark in column~8)
    \label{tab:results}}
\end{table}

These results --- summarized in Table~\ref{tab:results} ---
open the way for so-called ``difference of convex'' (DC)
optimization methods to achieve sparsity.
Indeed, the level sets of the \lzeropseudonorm\ can be expressed
by means of the difference between norms, taken from an increasingly or
decreasingly graded sequence of norms. 
And we provide a way to generate such sequences from a class of source norms that encompasses
the $\ell_p$ norms (but for the extreme ones).

To complete the possible applications, we add that, in another paper
\cite{Chancelier-DeLara:2021_SVVA},
we show that, with orthant-strictly monotonic norms, we can define
conjugacies for which  the \lzeropseudonorm\ is equal to its biconjugate.

\bigskip

\textbf{Acknowledgements.}
We thank 
Jean-Baptiste Hiriart-Urruty
for his comments on first versions of this work. 

\newcommand{\noopsort}[1]{} \ifx\undefined\allcaps\def\allcaps#1{#1}\fi

\end{document}